\documentclass{amsart}
\usepackage{amsmath, amsthm, amssymb, amscd, epsfig}
\usepackage{pinlabel}

\begin{document}

\newtheorem{theorem}{Theorem}[section]
\newtheorem{lemma}[theorem]{Lemma}
\newtheorem{proposition}[theorem]{Proposition}
\newtheorem{corollary}[theorem]{Corollary}
\newtheorem{conjecture}[theorem]{Conjecture}
\newtheorem{question}[theorem]{Question}
\newtheorem{problem}[theorem]{Problem}
\newtheorem*{claim}{Claim}
\newtheorem*{criterion}{Criterion}
\newtheorem*{rigidity_thm}{Random Rigidity Theorem~\ref{random_rigidity_theorem}}
\newtheorem*{norm_thm}{Random Norm Theorem~\ref{random_norm_theorem}}

\theoremstyle{definition}
\newtheorem{definition}[theorem]{Definition}
\newtheorem{construction}[theorem]{Construction}
\newtheorem{notation}[theorem]{Notation}

\theoremstyle{remark}
\newtheorem{remark}[theorem]{Remark}
\newtheorem{example}[theorem]{Example}

\numberwithin{equation}{subsection}

\def\N{\mathbb N}
\def\Z{\mathbb Z}
\def\R{\mathbb R}
\def\Q{\mathbb Q}
\def\H{\mathbb H}
\def\E{\mathcal E}
\def\M{\mathcal M}
\def\C{\mathcal C}
\def\T{\mathcal T}

\def\Pr{{\bf P}}
\def\Ex{{\bf E}}

\def\mean{\textnormal{mean}}
\def\var{\textnormal{var}}
\def\freq{\textnormal{freq}}
\def\cl{\textnormal{cl}}
\def\scl{\textnormal{scl}}
\def\homeo{\textnormal{Homeo}}
\def\rot{\textnormal{rot}}
\def\area{\textnormal{area}}
\def\vol{\textnormal{vol}}
\def\suff{\textnormal{suff}}

\def\Id{\textnormal{Id}}
\def\PSL{\textnormal{PSL}}
\def\til{\widetilde}

\def\uphi{\underline{\phi}}

\makeatletter
  \newcommand\tinyv{\@setfontsize\tinyv{5pt}{6pt}}
\makeatother

\title{Random rigidity in the free group}
\author{Danny Calegari}
\address{Department of Mathematics \\ Caltech \\
Pasadena CA, 91125}
\email{dannyc@its.caltech.edu}
\author{Alden Walker}
\address{Department of Mathematics \\ Caltech \\
Pasadena CA, 91125}
\email{awalker@caltech.edu}
\date{version 1.0; \quad \today}
\dedicatory{Dedicated to the memory of Andrew Lange}

\begin{abstract}
We prove a rigidity theorem for the geometry of the unit ball in random subspaces of
the $\scl$ norm in $B_1^H$ of a free group. In a free group $F$ of rank $k$, a random
word $w$ of length $n$ (conditioned to lie in $[F,F]$)
has $\scl(w) = \log(2k-1)n/6\log(n) + o(n/\log(n))$ 
with high probability, and the unit ball in a subspace spanned by $d$ random words
of length $O(n)$ is $C^0$ close to a (suitably affinely scaled) octahedron.

A conjectural generalization to hyperbolic groups and manifolds (discussed in the appendix)
would show that the length of a random geodesic in a hyperbolic manifold can be recovered
from the bounded cohomology of the fundamental group.
\end{abstract}

\maketitle

\section{Introduction}

Mostow's Rigidity Theorem says that a homotopy equivalence between closed
hyperbolic manifolds of dimension at least three is homotopic to an isometry. 
It follows that geometric invariants of a hyperbolic manifold have (at least in principle) a purely topological
definition. This is most apparent in Gromov's famous proof \cite{Gromov_manifolds} 
of the Rigidity Theorem which proceeds by showing that an obviously topological invariant ---
namely the Gromov (or $L^1$) norm of the fundamental class in homology --- is proportional to
the volume in any hyperbolic metric. 
As observed by Thurston \cite{Thurston_notes} a similar argument shows that for any
locally symmetric space $M$ modeled on a symmetric space $X$ there is a constant $C(X)$
so that the norm of the fundamental class $\|[M]\|_1$ satisfies
$$\|[M]\|_1 = C(X)\cdot\vol(M)$$
However, the determination of the constant $C(X)$ in any given case is extremely difficult.
Haagerup and Munkholm \cite{Haagerup_Munkholm} showed for $X$ equal to hyperbolic
$n$-space $\H^n$ that $C(\H^n)=1/v_n$ where $v_n$ is the volume of the regular ideal hyperbolic
$n$-simplex, and Bucher-Karlsson \cite{Bucher} showed that $C(\H^2\times \H^2)=3/2\pi^2$. The proofs are
very hard, and underscore the difficulty of computing the exact values of 
(nonzero) Gromov norms.

In this paper we prove a new kind of rigidity theorem for the $2$-dimensional relative Gromov
norm (or what is the same thing, the {\em stable commutator length} norm) 
in a free group $F$. This is a norm on a vector space $B_1^H(F)$, the homogenization
of the space $B_1$ of real group $1$-boundaries (in the bar complex). The space
$B_1^H$ is infinite dimensional, but its geometry can be probed by restricting attention
to finite dimensional subspaces. Our main theorem is a rigidity result for the geometry of
the unit ball in {\em random} finite dimensional subspaces of $B_1^H$ (technically: in subspaces
spanned by random elements of fixed length). We show that these
unit balls are (suitably scaled) $C^0$ close to {\em octahedra} 
(i.e. the unit ball in $\R^k$ with its usual $L^1$ norm). We also determine the {\em exact} 
scaling constant, and show that it has a simple expression in terms of the {\em growth
exponent} of the free group (i.e.\/ the entropy of the Markov process that generates random reduced words).
We concentrate in this paper on the case of free groups for clarity of exposition, but
similar results should hold for random words in arbitrary hyperbolic groups, or random geodesics in
negatively curved manifolds, with an analogous formula
for the scaling constant. We explain the idea of this generalization in an appendix, but
save the details for a follow-up paper.

\medskip

Recall that {\em stable commutator length} is an algebraic stabilization of the topological
notion of {\em filling genus}. If $X$ is a space, and $\Gamma:\coprod_i S^1 \to X$ is a
homologically trivial $1$-manifold, the {\em filling genus} of $\Gamma$ is the least
genus of a surface $S$ mapping to $X$ whose boundary represents the homotopy class of
$\Gamma$. The stable commutator length $\scl(\Gamma)$ is the infimum of $-\chi(S)/2n$ over
all $n$ and all surfaces $S$ mapping to $X$ whose boundary represents a cover $\hat{\Gamma}$
of $\Gamma$ of degree $n$. If $G$ is a group and $X$ is a space with $\pi_1(X)=G$, loops
in $X$ correspond to conjugacy classes in $G$, and the geometric definition given above
defines in a natural way a {\em pseudo-norm} on $B_1(G)$, the space of (real) $1$-boundaries;
i.e.\/ finite formal real linear combinations of elements in $G$ representing $0$ in
(real) homology. For $G$ a hyperbolic group, $\scl$ descends to a norm on 
a suitable homogenized quotient $B_1^H(G):=B_1/\langle g-hgh^{-1},g^n-ng\rangle$. Precise
definitions are given in \S~\ref{scl_background_section}.

Our first main theorem concerns the stable commutator length of a random element of
$[F,F]$ of prescribed length $n$ (we assume without comment
that $n$ is even, since a reduced element of
odd length is never in $[F,F]$). Here ``random'' means with respect to the uniform probability
on the finite set of reduced words of length $n$ in $[F,F]$ (when $n$ is even).
For clarity, we frequently use the standard Landau ``big $O$/little $o$''
notation, so the expression $O(g(x))$ denotes some function $f(x)$ satisfying
$f(x) \le C|g(x)|$ for some positive constant $C$ and for all $x\gg 0$,
the expression $\Theta(g(x))$ denotes some function $f(x)$
satisfying $C_1 g(x) \le f(x) \le C_2 g(x)$ for some positive constants $C_1,C_2$ and for all
$x \gg 0$, the expression $o(g(x))$ denotes some function $f(x)$ satisfying
$\lim_{x \to \infty} f(x)/g(x) = 0$, and so on. 
See e.g.\/ \cite{Knuth_bigO} for a reference.

\begin{rigidity_thm}
Let $F$ be a free group of rank $k$, and let $v$ be a random reduced element of length $n$,
conditioned to lie in the commutator subgroup $[F,F]$. Then for any $\epsilon>0$ and
$C>1$,
$$|\scl(v)\log(n)/n - \log(2k-1)/6| \le \epsilon$$
with probability $1-O(n^{-C})$.
\end{rigidity_thm}
In particular, this implies that $\scl(v)\log(n)/n$ converges in probability to $\log(2k-1)/6$
as $n\to \infty$.

In more geometric language, we derive strong control on the geometry of the unit ball in the
$\scl$ norm in a random subspace.

\begin{norm_thm}
Let $F$ be a free group of rank $k$, and for fixed $d$, let $v_1,v_2,\cdots,v_d$ be 
independent random
reduced elements of length $n_1,n_2,\cdots,n_d$ conditioned to lie in $[F,F]$, where
without loss of generality we assume $n_1\ge n_i$ for all $i$. Let $V$ be the subspace of
$B_1^H(F)$ spanned by the $v_i$. Then for any $\epsilon>0, C>1$ and real numbers $t_i$,
$$|\scl(\sum t_iv_i)\log(n_1)/n_1 - \log(2k-1)(\sum |t_i|n_i)/6n_1|\le \epsilon$$
with probability $1-O(n_1^{-C})$.
\end{norm_thm}
In words: the unit ball in the $\scl$ norm scaled by $n_1/\log(n_1)$ converges to the
unit ball in the norm $\|\sum t_i v_i\| = \sum |t_i|n_i/n_1$ in the $C^0$ topology and
in probability, as $n_1 \to \infty$. If $n_i=n_1+o(n_1)$ for all $i$, the unit ball is $C^0$
close to a (scaled) octahedron.

\medskip

It is worth remarking that the speed of convergence is very slow. Our asymptotic theorems depend
on the distribution of the subwords of a random word at a particular characteristic scale: for a word of
length $n$, we focus on the subwords of length $O(\log(n))$. There are some ``boundary effects'' 
which suggest a heuristic correction to our asymptotic formula which becomes insignificant only
when $\log(n)$ is sufficiently large. Computer experiments (described in \S~\ref{heuristic_section})
show this heuristic correction to be in very good agreement with reality. However we are not
able to rigorously justify this observation nor obtain a precise asymptotic estimate of the 
error.

\subsection{Acknowledgments}
We would like to thank Jeremy Kahn and Richard Sharp
for some useful conversations about this material.
We would also like to thank the anonymous referee for helpful comments and suggestions.
Danny Calegari was supported by NSF grant DMS 1005246. 

\section{The random reduced word}\label{random_word_section}

\subsection{Reduced words}

Fix a free group $F$ of rank $k$ and a free generating set. The generators
will be denoted $a$, $b$, $c$ and so on, and their inverses by 
$A$, $B$, $C$.

We are interested in random reduced words conditioned to lie in the commutator
subgroup. This is a complicated (non-local) condition to impose
on a word. Fortunately, there is a nice estimate, due to Sharp, of the relative
proportion of words of length $n$ in $[F,F]$.

\begin{theorem}[Sharp \cite{Sharp}, Thm.~1]\label{Sharp_theorem}
Let $F$ be a free group of rank $k\ge 2$. Let $F_n$ denote the set of elements of $F$ of
length $n$, and let $F_n'=F_n\cap [F,F]$. If $n$ is odd, $F_n'$ is empty, whereas
there is an explicit constant $\sigma$ depending on $k$ so that 
$$\lim_{n\to\infty, \; n\text{ even}} \left|\sigma^k n^{k/2} \frac {|F_n'|}{|F_n|} - \frac
{2} {(2\pi)^{k/2}}\right|=0$$
where the limit is taken over {\em even} positive integers $n$.
\end{theorem}

This theorem has the following consequence. Suppose that a random element of $F_n$
has some property $P$ with probability $1-o(n^{-k/2})$. Then a random element
of $F_n'$ has property $P$ with probability $1-o(1)$.
In practice, we are interested in properties of random elements 
in $F_n$ that hold with probability $1-O(C^{-n^c})$ for some constants $C>1,c>0$, or
with probability $1-O(n^{-C})$ for all $C>0$, and Sharp's theorem is the fundamental
tool that lets us draw conclusions about random elements of $F_n'$.

\medskip

In the sequel we use the following notation consistently, where possible.
We let $v$ denote a random reduced word of length $n$, and let 
$m(n,k)$ (or just $m$ for brevity) be defined by $m(n,k):=\log(n)/\log(2k-1)$. 
There is a stationary Markov process which produces random reduced words in $F$ 
with the uniform probability, and $\log(2k-1)$ is the entropy of this process.

\subsection{Phase transition}

The constant $m=\log(n)/\log(2k-1)$ is a {\em natural length scale} on which to 
view subwords of a random word of length $n$. A random word of length $100$ like
{\par\noindent \, \tinyv \it  bbbbaBAbAABaBaabbabbaBAABBAABBAbabAAbbABBBAbaaaaBAAbbABaBabaBaBAbAABBBBaBabbaaBAAABaBabAbABaaaabbbAA}
\par\noindent does not look homogeneous to the naked eye; the long strings of capital letters
leap out and draw the reader's attention to specific locations in the word. 
The meaning of the scale $m$ is that
a random word of length $n$ (for sufficiently large $n$) looks {\em homogeneous} on scales
smaller than $m$, and {\em heterogeneous} on scales larger than $m$. However for this phase
transition to become truly apparent, one must take $n$ very large, so that $m\sim\log(n)\gg 1$.

One way to quantify this distinction is to fix a length $\ell$ and compute some statistic 
associated to the set of subwords of $v$ of length $\ell$. Each subword is an element of
$F_{\ell}$ (the set of elements of $F$ of length $\ell$), and a natural number to count is 
$$A_\ell(v):= \frac 1 {2n} \sum_{w\in F_\ell} |\,\text{copies of } w \text{ in } v - \text{copies of } w^{-1}
\text{ in } v \,|$$
If $v$ is cyclically reduced, and one counts copies in the cyclic word $v$, then $A_\ell(v)\in [0,1]$
with $A_\ell(v)=1$ if and only if no inverse pair of subwords of length $\ell$ appear.
There is a {\em phase transition} in $A_\ell$: for $\ell = Lm$ for some fixed $L<1$ we have $A_\ell(v) \to 0$ in probability, whereas for 
$\ell = Lm$ for some fixed $L>1$ we have $A_\ell(v) \to 1$. This is proved in 
\S~\ref{L<1subsection}--\ref{L>1subsection}.

For words of length $n=10000$ in rank $k=2$ we have $\log(n)/\log(2k-1)\approx 8.383613$. We
compute $A_\ell(v)$ for a random word $v$ in $[F,F]$ of length $10000$ for $1\le \ell\le 11$ 
(there are $236196$ reduced words of length $11$). This data is presented in 
Figure~\ref{cancellation_histogram_figure}. Note that conditioning $v$ to lie in $[F,F]$
forces $A_1(v)=0$. The figure hints at a phase transition at $\ell \sim m$ but for it to be
really sharp, one would need to take something like $n\sim \text{googol}$.

\begin{figure}[htpb]
\labellist
\small\hair 2pt
\pinlabel $A_\ell(v)$ at -18 40
\pinlabel $0$ at -5 15
\pinlabel $1$ at -5 65
\pinlabel $\ell$ at 320 16
\pinlabel $1$ at 5 5
\pinlabel $6$ at 155 5
\pinlabel $11$ at 305 5
\pinlabel $m$ at 227 2
\endlabellist
\centering
\includegraphics[scale=1]{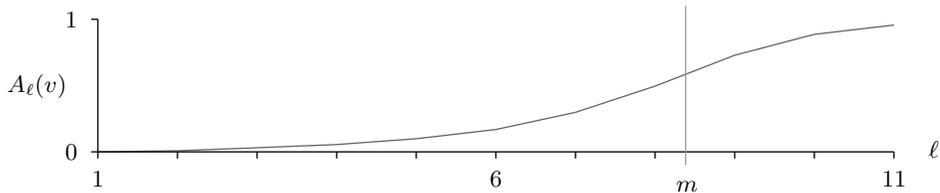}
\caption{Values of $A_\ell(v)$ for $v$ a random word of length $10000$ and
$1\le \ell \le 12$.}\label{cancellation_histogram_figure}
\end{figure}

\subsection{Counting functions and counting measures}\label{counting_function_subsection}

We use the notation $F_i$, $F_{<i}$, $F_{\ge i}$ and so on for the set of
elements in $F$ of length $i$, $<i$, $\ge i$ respectively. A {\em random word} 
of length $n$ is an element of $F_n$, chosen with the uniform probability measure.
Note that the cardinality of $F_n$ is $(2k)(2k-1)^{n-1}$, so $|F_{Lm}| \sim n^L$.

Although it does not add much technically, we think of $F$ as a measure space, 
with the Borel algebra consisting of all subsets. Consequently any function 
$f$ on $F$ is measurable, and a function $f$ is in $L^1(F)$ if and only if 
$\sum_{g\in F} |f(g)| < \infty$. 

\begin{definition}
For a reduced word $\sigma$, the {\em counting function} $C_\sigma$ is defined by
$$C_\sigma(v) = \text{number of copies of }\sigma\text{ in }v$$
and the {\em counting measure} $C(v)$ is the measure  on $F$ of total
mass $|v|(|v|-1)/2$ for which $C(v)(\sigma) = C_\sigma(v)$.

For $f$ a measurable function on $F$, define
$$C_f(v) = \int_F f dC(v)$$
and define $H_f(v):=C_f(v) - C_f(v^{-1})$.
\end{definition}

\subsection{Accurately estimating $C_\sigma(v)$}\label{L<1subsection}

If $v$ is a random word of length $n$, and $\sigma$ is a random word of length $Lm$ where
$L<1$, we need to estimate $C_\sigma(v)$. Since $v$ contains $n-|\sigma|+1$
subwords of length $|\sigma|$, the ``expected'' number of copies of $\sigma$ in $v$
is $(n-|\sigma|+1)/|F_{|\sigma|}|= n^{1-L}(2k)/(2k-1) \pm O(\log(n))$.
If subwords were independent, one would expect the deviation from this expected value
to be typically of order $n^{(1-L)/2}$, and to be of order $n^{\epsilon+(1-L)/2}$ only
with exponentially vanishing probability. This is what we prove:

\begin{proposition}\label{small_error}
Let $L<1$. Then for any $\epsilon>0$ there are constants $C>1$ and $c>0$ so that
$$\Pr\left(|C_\sigma(v) - n/|F_{Lm}||<n^{\epsilon+(1-L)/2}\text{ for all }\sigma \in F_{Lm}\right) = 1-O(C^{-n^c})$$
\end{proposition}
\begin{proof}
The strategy is as follows.
We first show that for each fixed word $\sigma$ of length $Lm$ the
inequality
$\Pr\left(|C_\sigma(v) - n/|F_{|\sigma|}||<n^{\epsilon+(1-L)/2}\right) = 1-O(C^{-n^c})$
holds.
Since there are only $O(n^L) < n$ words of length $Lm$, it will follow
that the desired estimate will hold for {\em every} $\sigma\in F_{Lm}$ with probability
$1-O(nC^{-n^c})$. Absorbing the $n$ factor into the constants $C$ and $c$, we will
be done.

Choose some constant $N$ (we will decide on the exact value of $N$ later). 
For each residue class $j$ mod $Nm$, let $v_{j,i}$ be the subword of
$v$ of length $Lm$ which starts at the $j+iNm$th letter of $v$. The point is that 
for fixed $j$, the $v_{j,i}$ for consecutive $i$ 
are ``almost'' independent. This is made precise in the following
lemma:

\begin{lemma}\label{decay_of_correlation}
For any two words $x$, $y$ of length $|\sigma|$, there is an inequality
$$|\Pr\left(v_{j,i}=x \; | \; v_{j,i-1}=y\right) - 1/|F_{|\sigma|}|| \le (2k-2)^{-(N-1)m}$$
\end{lemma}
\begin{proof}
Let $yuz$ be the subword of $v$ starting at $y$, where $u$ has length $(N-1)m$.
The number of words $u$ of fixed length for which $yuz$ is reduced depends only on the length of $u$,
the last letter of $y$, and the first letter of $z$. For any single letters $a,b$ we let
$u_m(a,b)$ denote the number of reduced words of the form $aub$ of length $m+2$.
We show by induction on $m$ that the following two statements are true:
\begin{enumerate}
\item{$u_m(a,b)=u_m(a,c)$ if neither of $b,c$ are equal to $a^{-1^{m+1}}$}
\item{$|u_m(a,a^{-1^{m+1}})/u_m(a,b) - 1| \le (2k-2)^{-m}$}
\end{enumerate}
Since $u_1(a,b)=(2k-2)$ if $b\ne a$ and $u_1(a,a) = (2k-1)$ this is true for $m=1$.

Assume it is true for $(m-1)$ odd (for example). Then depending on the first letter
of $u$ we have two cases (by the induction step), and we deduce
$$u_m(a,b) = u_{m-1}(b,b) + (2k-2)u_{m-1}(c,b) \text{ for } b\ne A, c\ne b$$
$$u_m(a,A) = (2k-1)u_{m-1}(c,A) \text{ for } c\ne A$$
and the induction step is proved. The case $(m-1)$ even is analogous. The lemma follows.
\end{proof}

We resume the proof of Proposition~\ref{small_error}. By Lemma~\ref{decay_of_correlation},
the probability that $v_{j,i}=\sigma$ conditioned on the value of $v_{j,i-1}$ is very nearly
independent of the value of $v_{j,i-1}$, so we can compare 
the number of $\sigma$s
among the $v_{j,i}$ (for fixed $j$) with a sum of independent 
Bernoulli variables, and estimate the deviation from the mean using
the Chernoff bound.
Let $C_{\sigma,j}(v)$ be the number
of copies of $\sigma$ among the $v_{j,i}$.

\begin{lemma}
Suppose $N\ge 3$. For each $j$, and for any positive $\epsilon$, there is an inequality
$$\Pr\left(|C_{\sigma,j}(v) - n/(Nm\cdot|F_{|\sigma|}|)| > n^{\epsilon + (1-L)/2} \right) = O(C^{-n^c})$$
\end{lemma}
\begin{proof}
By Lemma~\ref{decay_of_correlation}, the conditional probability
that successive $v_{j,i}$ are equal to $\sigma$ is never more than 
$1/|F_{|\sigma|}| + (2k-2)^{-(N-1)m}$, or less than
$1/|F_{|\sigma|}| - (2k-2)^{-(N-1)m}$.
So we can bound the probability of a large deviation in terms of such large deviations
for sums of independent Bernoulli trials. 

Since $(2k-2)^{-(N-1)m} \le n^{-0.6(N-1)}$
(using the estimate $\log(2k-2)/\log(2k-1)>0.6$ for $k\ge 2$), 
when $N\ge 3$ we have $(2k-2)^{-(N-1)m} \le n^{-1}$. 

We have the Chernoff bound (e.g.\/ the upper bound in Thm.~1.3.13 from \cite{Stroock})
$$\Pr(|S_n -np| \ge \delta np) \le e^{-\delta^2np/3}$$
where $S_n$ is a sum of $n$ independent Bernoulli random variables with parameter $p$.
Using $p_+=n^{-L}(2k-1)/(2k) + n^{-0.6(N-1)}< n^{-L}$, we obtain
$$\Pr(C_{\sigma,j}(v) - n/(Nm\cdot|F_{|\sigma|}|) - n^{1-0.6(N-1)}/Nm \ge \delta np_+/Nm) \le e^{-\delta^2np_+/3Nm}$$
Since $Nm=O(\log(n))$ and $n^{1-0.6(N-1)}/Nm<1$, 
taking $\delta=n^{\epsilon - (1-L)/2}Nm$ this implies
$$\Pr(C_{\sigma,j}(v) - n/(Nm\cdot|F_{|\sigma|}|) \ge n^{\epsilon + (1-L)/2}) \le O(C^{-n^c})$$
where $C>1$, $c>0$ depend only on $\epsilon$.

A similar inequality holds for $n/(Nm\cdot|F_{|\sigma|}|) - C_{\sigma,j}(v)$.
\end{proof}

We now complete the proof of Proposition~\ref{small_error}.
Since $j$ was arbitrary, it follows that {\em every} $C_{\sigma,j}(v)$ deviates from
$n/(Nm\cdot|F_{|\sigma|}|)$ by at most $n^{\epsilon + (1-L)/2}$, with probability at
least $1-Nm\cdot O(C^{-n^c})$ which is still $1-O(C^{-n^c})$. Hence
$C_\sigma(v) = \sum_j C_{\sigma,j}(v)$ deviates from $n/|F_{|\sigma|}|$ by at most
$Nm\cdot n^{\epsilon + (1-L)/2} < n^{\epsilon' + (1-L)/2}$ with the same probability.
The proposition follows.
\end{proof}

In Appendix~\ref{Markov_chain_section}, 
we compare this result with Chernoff-type inequalities for
nonreversible Markov chains obtained by Lezaud, Dinwoodie and others, and interpret
such bounds in terms of the Cheeger constants of certain directed graphs.

\subsection{Bounding $\sum_\sigma C_\sigma(v^{-1})$}\label{L>1subsection}

We now turn our attention to words of length $>m$. Fix some $L>1$, and let $S$ 
be the set of subwords of $v$ of length $Lm$.

\begin{proposition}\label{long_inverse_word_bound}
For any $\epsilon$ there are constants $C>1$ and $c>0$ so that
$$\Pr\left( \sum_{\sigma \in S} C_\sigma(v^{-1}) < n^{2-L+\epsilon} \right) = 1 - O(C^{-n^c})$$
In particular, for $\epsilon < L-1$,
with probability $1-O(C^{-n^c})$ there is a subset $S'$ of $S$ with
$$\text{card}(S - S') < n^{2-L+\epsilon} = o(n/\log(n))$$ 
so that no element $\sigma \in S'$ appears in $v^{-1}$.
\end{proposition}

\begin{remark}
Note that we think of $S$ just as a set, {\em not} a set with multiplicity. For applications,
it will be important to show that the cardinality of $S$ is close to $n$ with probability
$1-O(C^{-n^c})$; we show this as Proposition~\ref{large_cardinality}.
\end{remark}

\begin{remark}
The set of words of length $Lm$ has cardinality of order $n^L$, so the subset $S$ has
measure of order $n^{1-L}$. If we fix in advance any subset $S$ of $F_{Lm}$ of measure $n^{1-L}$,
a robust Chernoff-type bound for Markov chains due to Lezaud (see 
Appendix~\ref{Markov_chain_section}) gives a bound on $\sum_{\sigma \in S} C_\sigma(v^{-1})$.
However this estimate cannot be applied naively to our context, since $S$ depends (very
strongly) on $v$.
\end{remark}

\begin{proof}
It is awkward to find a purely probabilistic proof of this estimate, because
overlapping subwords of $v$ are necessarily very highly correlated. The non-proba{-}bilistic
ingredient in our proof is the following simple, but important observation:
\begin{lemma}\label{no_overlap_inverse}
Let $v$ be a reduced word. Then for any reduced word $\sigma$, no copy of 
$\sigma$ in $v$ can overlap a copy of $\sigma^{-1}$.
\end{lemma}
\begin{proof}
If $\sigma$ overlaps $\sigma^{-1}$, then without loss of generality we can write $\sigma$
as $xy$ where $y=y^{-1}$. But this is absurd.
\end{proof}
Now, for each $i$, let $v_i$ be the subword of $v$ of length $Lm$ starting at the $i$th
letter, and let $v_{<i}$ and $v_{>i}$ denote the part of $v$ outside $v_i$, so that
$v=v_{<i}v_iv_{>i}$ as a reduced word. Further, let $S_{<i}$ (resp. $S_{>i}$) denote
the subset of $S$ consisting of subwords of length $Lm$ in $v_{<i}$ (resp. $v_{>i}$). By
Lemma~\ref{no_overlap_inverse},
$$\sum_{\sigma \in S} C_\sigma(v^{-1}) = \sum_i \sum_{\sigma \in S_{<i}} C_\sigma(v_i^{-1}) + \sum_i \sum_{\sigma \in S_{>i}} C_\sigma(v_i^{-1})$$
The point is that we can bound $\sum_{\sigma \in S_i} C_\sigma(v_i^{-1})$ in probability
conditioned on $v_{<i}$, {\em independently} of $v_{<i}$.
\begin{lemma}
For any $\epsilon$,
$$\Pr\Bigl( \sum_{\sigma \in S_{<i}}C_\sigma(v_i^{-1}) = 1 \; | \; v_{<i} \Bigr) < n^{1-L+\epsilon}$$
\end{lemma}
\begin{proof}
Note that $\sum_{\sigma \in S_{<i}}(v_i^{-1})$ is $1$ or $0$, depending on whether $v_i^{-1}$ is
in the set $S_{<i}$ or not. No matter what $v_{<i}$ is, there are $(2k-1)^{Lm}>n^{L-\epsilon}$ choices
for $v_i$, and each occurs with the uniform probability. The cardinality of $S_{<i}$ is at most
$i$ which is less than $n$, so the chance that $v_i^{-1}$ is in $S_{<i}$ is at most
$n^{1-L+\epsilon}$, as claimed.
\end{proof}
It follows that if we fix a residue $j$ mod $Lm$, for any $\epsilon$ there are $C>1$, $c>0$ such that
we can estimate
$$\Pr\Bigl( \sum_{i=j\text{ mod }Lm} \sum_{\sigma \in S_{<i}} C_\sigma(v_i^{-1}) \ge n^{2-L+\epsilon}/Lm \Bigr)
< O(C^{-n^c})$$
Summing over all residue classes $j$, and then replacing $S_{<i}$ by $S_{>i}$ by symmetry proves
the proposition.
\end{proof}

As remarked above, it is important for applications to show that the cardinality of $S$ is very
close to $n$, with high probability.

\begin{proposition}\label{large_cardinality}
Fix $L>1$ and let $S$ denote the set of subwords of $v$ of length $Lm$.
There is an $\epsilon$ and $C>1$, $c>0$ so that
$$\Pr(n-\text{card}(S) > n^{1-\epsilon}) = O(C^{-n^c})$$
\end{proposition}
\begin{proof}
The proof is almost the same as that of Proposition~\ref{long_inverse_word_bound}, except
that we need to estimate the number of $\sigma$ for which some copy of $\sigma$ overlaps
itself, and show this is $<n^{1-\epsilon}$ for some $\epsilon$ with the desired probability.

There are two kinds of overlaps to consider: those for which the nonoverlapping initial
segment of the first word has length $<2m/3$ (``big overlaps'') and those for which
it has length $\ge 2m/3$ (``little overlaps''). We count the number of
each independently. 

A big overlap results in a subword of the form $wuw$ where the length of $w$ is at least $m/6$
and the length of $u$ is at least $m/6$. Conditioned on $w$ and $u$, the probability that
the next word will be a copy of $w$ is at most $n^{-1/6}$, so there are at most $n^{5/6}$
subwords of $v$ that are contained in a big overlap. 
A little overlap results in a subword of the form $ww$ where the length of $w$
is at least $2m/3$. Again, conditioned on $w$, the probability that the next word will be a copy
of $w$ is at most $n^{-2/3}$ so there are at most $n^{1/3}$ subwords of $v$
that are contained in little overlaps. Each subword is contained in at most
$Lm = O(\log(n))$ overlaps of either kind. The result follows.
\end{proof}

\section{Stable commutator length}\label{scl_background_section}

The material in this section is standard. A basic reference is \cite{Calegari_scl}.

\subsection{Definitions}

\begin{definition}
Let $G$ be a group, and $[G,G]$ the commutator subgroup. The {\em commutator length}
of an element $g\in [G,G]$, denoted $\cl(g)$, is the least number of commutators whose
product is $g$; and the {\em stable commutator length}, denoted $\scl(g)$, is the limit
$\scl(g):=\lim_{n\to\infty} \cl(g^n)/n$.
\end{definition}

The definition of (stable) commutator length can be extended to finite formal sums as follows:

\begin{definition}
Let $G$ be a group, and let $\lbrace g_i \rbrace$ be a finite collection of elements with
$\prod_i g_i \in [G,G]$. Define $\cl(\sum g_i)$ to be the minimum of $\cl(\prod g_i^{h_i})$
over all products of conjugates $g_i^{h_i}$ of the $g_i$. This is symmetric, and a class
function in each $g_i$ separately. Define $\scl(\sum g_i) = \lim_{n\to\infty} \cl(\sum g_i^n)/n$.
\end{definition}

Let $C_1(G)$ be the real vector space with basis the elements of $G$, and let $B_1(G)$ be
the kernel of $C_1(G) \to H_1(G;\R)$. So $B_1(G)$ is the space of formal finite real linear
combinations of elements in $G$ that represent $0$ in (real) homology. Equivalently, $B_1(G)$
is the image of the vector space of real $2$-chains (in the bar complex) under $\partial$.
It is a fact that $\scl$ extends by linearity and continuity to a pseudo-norm on $B_1(G)$,
and vanishes on the subspace $\langle g-hgh^{-1}, g^n - ng\rangle$. This vanishing reflects
the homogeneity of $\scl$ and the fact that it is a class function in each variable separately.
So $\scl$ descends to a pseudo-norm on the quotient $B_1^H(G):=B_1(G)/\langle g-hgh^{-1}, g^n-ng\rangle$.

The following theorem is nice to know, but is not used in an essential way in this paper:
\begin{theorem}[Calegari-Fujiwara \cite{Calegari_Fujiwara}]
Let $G$ be (word) hyperbolic. Then $\scl$ is a norm on $B_1^H(G)$.
\end{theorem}

\subsection{Surfaces}

Let $X$ be a space with $\pi_1(X)=G$, and for any finite collection of conjugacy classes
$g_i$ let $\Gamma:\coprod_i S^1_i \to X$ be a $1$-manifold in the associated free homotopy
class. A map of a (compact, oriented) surface $f:S \to X$ is {\em admissible} if there is
a commutative diagram
$$\begin{CD}
\partial S @>>> S \\
@V\partial f VV @Vf VV \\
\coprod_i S^1_i @>\Gamma >> X
\end{CD} $$
and an integer $n(S)$ for which $\partial f_*[\partial S] = n(S)[\coprod_i S^1_i]$
in $H_1$. The map is {\em monotone} if $\partial S \to \coprod_i S^1_i$ is homotopic to
an orientation-preserving cover (equivalently, if every component of $\partial S$ wraps with
positive degree around its image).

\begin{lemma}[\cite{Calegari_scl}, Prop.~2.74]\label{surface_lemma}
Let $g_1,\cdots,g_m$ be conjugacy classes in $G$, represented by $\Gamma:\coprod_i S^1_i \to X$.
Then
$$\scl(\sum_i g_i) = \inf_S \frac {-\chi^-(S)} {2n(S)}$$
where the infimum is taken over all surfaces $S$ and all maps $f:S \to X$ admissible for
$\Gamma$.
\end{lemma}

The notation $\chi^-(S)$ means the sum of Euler characteristics $\sum_i\chi(S_i)$ taken over those 
components $S_i$ of $S$ with $\chi(S_i)\le 0$. By \cite{Calegari_scl}, Prop.~2.13 it suffices to
restrict to {\em monotone} admissible surfaces. An admissible surface $S$ is 
{\em extremal} if equality is achieved.

\subsection{Fatgraphs}\label{fatgraph_subsection}

If $F$ is free, $X$ can be taken to be a graph, and any admissible surface can be
represented combinatorially (possibly after performing some compressions) 
by a {\em fatgraph}. Fatgraphs are combinatorial objects which allow one to move
back and forth between group theory/combinatorics and 2-dimensional topology; a standard
reference is \cite{Penner}, especially \S~1.

A {\em fatgraph} $Y$ is a graph together with a cyclic ordering of the edges 
incident at each vertex.
Such a graph can be thickened to a compact surface $S(Y)$ (or just $S$ if $Y$ 
is understood) in such a way that $Y$ embeds in $S(Y)$ as a
deformation retract. A fatgraph $Y$ is {\em oriented} if $S(Y)$ is oriented. In the sequel we
assume all our fatgraphs are oriented, and have no $1$-valent vertices. 
Note that $\chi(Y)=\chi(S(Y))$.

A {\em fatgraph over $F$} is a fatgraph with oriented edges labeled by words in $F$
so that opposite sides get inverse labels, and the cyclic words obtained by reading 
around $\partial S(Y)$ are reduced. By abuse of notation we write $\partial Y$ in place
of $\partial S(Y)$ and think of it as an element of $B_1^H(F)$. Figure~\ref{fatgraph_figure} 
gives an example of an extremal fatgraph for the chain $a+b+AB+[a,b]$ in $F_2$. Note that
extremal surfaces do not need to be connected.

\begin{figure}[htpb]
\labellist
\small\hair 2pt
\pinlabel $a$ at 188 472
\pinlabel $A$ at 160 454
\pinlabel $B$ at 132 439
\pinlabel $b$ at 112 416
\pinlabel $B$ at 220 400
\pinlabel $b$ at 220 365
\pinlabel $a$ at 345 400
\pinlabel $A$ at 375 420
\pinlabel $A$ at 155 360
\pinlabel $a$ at 178 342
\pinlabel $b$ at 85 250
\pinlabel $B$ at 70 220
\pinlabel $B$ at 221 260
\pinlabel $b$ at 250 257
\pinlabel $b$ at 280 250
\pinlabel $B$ at 308 248
\pinlabel $A$ at 270 95
\pinlabel $a$ at 295 110 
\pinlabel $a$ at 315 130
\pinlabel $A$ at 338 150
\pinlabel $a$ at 380 290
\pinlabel $A$ at 400 315
\pinlabel $b$ at 419 146
\pinlabel $B$ at 442 125
\endlabellist
\centering
\includegraphics[scale=0.5]{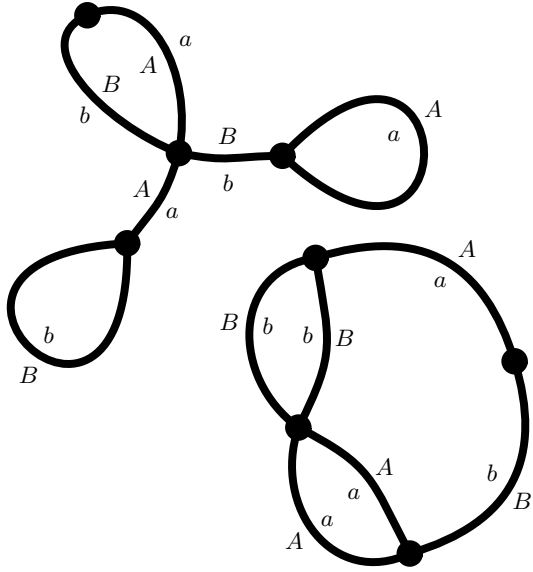}
\caption{An extremal surface for $a+b+AB+[a,b]$ represented as a fatgraph.}\label{fatgraph_figure}
\end{figure}

The basic fact we use is the following lemma, which is a restatement of \cite{Culler}, Thm.~1.4 in
the language of fatgraphs. 

\begin{lemma}[Culler \cite{Culler}, Thm.~1.4 (fatgraph lemma)]\label{fatgraph_lemma}
Let $S$ be an admissible surface bounding a chain $\Gamma$. Then after possibly compressing
$S$ a finite number of times (thereby reducing $-\chi^-(S)$ without changing $\partial S$)
there is a fatgraph $Y$ over $F$ with $S(Y)=S$ and $\partial Y=\Gamma$.
\end{lemma}

\begin{remark}
Culler proves his theorem only for surfaces with
connected boundary, but his argument generalizes with no extra work. 
An equivalent statement, valid for surfaces with disconnected boundary,
is also proved in \cite{Calegari_rational}, Lem.~3.4; also
see \cite{Calegari_scl} \S~4.3 for a discussion and references.
\end{remark}

Let $Y$ be an extremal fatgraph for $v$. The underlying fatgraph might not be trivalent, but by
splitting higher valence vertices, and inserting (unlabeled) ``dummy edges'', we can think of
$Y$ as a trivalent fatgraph in a degenerate way, where some degenerate ``edges'' have length $0$.
We call this the operation of {\em resolving vertices} (such a resolution need not be unique).

\begin{lemma}\label{combinatorial_estimate}
Let $Y$ be an extremal fatgraph for $v$, so that $\partial Y$ represents $Nv$ for some $N$,
and $-\chi(Y)/2N = \scl(v)$. Resolve vertices of $Y$ so that $Y$ is trivalent,
possibly with some edges of length $0$. Let the average length of the edges of $Y$ be
$\ell m$. Then
$$\scl(v) = n\log(2k-1)/12\ell\log(n)$$
\end{lemma}
\begin{proof}
Suppose $Y$ has $V$ vertices and $E$ edges. Since $Y$
is trivalent, $2E/3=V$ and $-\chi(Y)=E-V = E/3$. On the other hand, the total length
of $\partial Y$ is $Nn = 2E\ell m$. Hence 
$$\scl(v) = -\chi(Y)/2N = E/6N = n/12\ell m = n\log(2k-1)/12\ell\log(n)$$
\end{proof}
It will be our goal to show that for {\em random} $v$ of length $n\gg 1$, 
the extremal fatgraph $Y$ has $\ell = 1/2 + o(1)$ with probability $1-o(1)$.

\begin{remark}
The reader who is unhappy with edges of length $0$ can just take $\ell m$ to be equal to the
total length of $Y$ divided by $E+\sum_v \text{valence}(v)-3$.
\end{remark}

\section{Random values of scl}\label{random_value_section}

The goal of this section is to prove the Random Rigidity Theorem:

\begin{theorem}[Random Rigidity Theorem]\label{random_rigidity_theorem}
Let $F$ be a free group of rank $k$, and let $v$ be a random reduced element of length $n$,
conditioned to lie in the commutator subgroup $[F,F]$. Then for any $\epsilon>0$ and
$C>1$,
$$|\scl(v)\log(n)/n - \log(2k-1)/6| \le \epsilon$$
with probability $1-O(n^{-C})$.
\end{theorem}

The proof will occupy most of the remainder of the section.

\subsection{Upper bounds}

The upper bound in the Random Rigidity Theorem is sharpened by the following
proposition:

\begin{proposition}\label{upper_bound_proposition}
Let $v$ be a random reduced word in the commutator subgroup of length $n$.
Then for any $\epsilon > 0$ there are constants $C>1$ and $c>0$ so that
$$\scl(v)\log(n)/n - \log(2k-1)/6 \le \epsilon$$
with probability $1-O(C^{-n^c})$.
\end{proposition}

Given random $v$, we explicitly build an extremal surface 
(actually an extremal fatgraph) by gluing together a very large
number of {\em tripods} with edges of length slightly less than $(1-\epsilon)m/2$. 
The fact that such tripods can be glued up to produce a fatgraph with
boundary very close to a multiple of $v$ follows from an equidistribution lemma, derived from
the estimates in \S~\ref{random_word_section}, which holds with very high probability for
most random $v$. The tripods do not glue up completely, but the mass of the
unglued part has size $O(n^{-\epsilon/2})$ compared to the glued part, and the
remainder can be glued up (under the hypothesis that $v$ is homologically trivial)
with a contribution to $\chi$ proportional to the mass. 

\subsection{Tripods and joints}

In what follows we generally adhere to the notational convention that group inverses are
denoted by small and capital letters; hence $X$ means $x^{-1}$ and so on.

\begin{definition}
A {\em tripod of edge length $L$} is a fatgraph with underlying graph a tripod, and
with edges labeled by reduced words $xY$, $yZ$, $zX$ where each of $x$, $y$, $z$  
(the {\em incoming edge labels}) has length $L$. We denote such a tripod $T(x,y,z)$

A {\em copy} of $T(x,y,z)$ is a triple of segments
of the form $xY$, $yZ$, $zX$ in $v$. These segments may appear anywhere in $v$; they might or
might not be adjacent, and are allowed to overlap each other.
\end{definition}

\begin{lemma}
A triple $x,y,z$ of reduced words of length $L$
are the labels of a tripod if and only if their last letters are distinct.
Consequently, for any reduced word $xY$ of length $2L$, there are $(2k-2)(2k-1)^{L-1}$ choices
for $z$.
\end{lemma}
\begin{proof}
Obvious.
\end{proof}

There are $(2k)(2k-1)(2k-2)(2k-1)^{3(L-1)}/3 \sim (2k-1)^{3L}/3$ tripods $T$ of edge length $L$. 
For each tripod $T$, let $\partial T$ denote the triple of words $xY, yZ, zX$. 

\begin{definition}
A {\em joint of edge length $L$} is a fatgraph with underlying graph a segment, 
and with edges labeled by reduced words $x$, $X$ each of length $L$. Denote such a joint
$J(x)$.

A {\em copy} of $J(x)$ is an {\em ordered} pair of segments of the form $x$, $X$ in $v$.
Again, these segments may appear anywhere in $v$ (note that since $v$ is reduced, these
segments cannot overlap or be adjacent in $v$).
We distinguish between orientations, so that $J(x)$ and $J(X)$ are different. 
\end{definition}

Each joint $J(x)$ is contained in a unique maximal joint $J(x')$.

Fix $L$ with $L/m=1/2-\epsilon$ for some small $\epsilon$.
For a word $v$, let $T_L(v)$ denote the set of copies of tripods of edge length $L$ in $v$,
and let $J_L(v)$ denote the set of copies of joints of edge length $L$ in $v$. Note that
each pair of subwords $x,X$ of length $L$ in $v$ determines {\em two} elements of $J_L(v)$.
We define an involution $\iota$ on the set $J_L(v)$ interchanging such pairs.
If $v$ is understood, we just write $T_L$ and $J_L$.

Given $T$, a copy of $T(x,y,z)$ of length $L$, there are three associated joints
$J(x)$, $J(y)$, $J(z)$ which can be extended uniquely to maximal joints $J(x')$, $J(y')$
and $J(z')$. Note that $x$ is a suffix of $x'$, and so on.
Define $\partial T(x,y,z) = J(x') + J(y') + J(z')$ and extend $\partial$ to a linear
map from the space of measures on $T_L$ to the space of measures on $J_L$.

\begin{example}
Let $v=ABBAbAABAAbababaabbABBBBabbaaB$. The tripod of length $2$ as indicated:
$$ABB\underline{AbAA}BAAbabab\underline{aabb}AB\underline{BBBa}bbaaB$$
is associated to three joints: a pair $Ab,Ba$; a pair $AA,aa$; and a pair $bb,BB$.
The joint $Ab,Ba$ is contained in a maximal joint of length $5$:
$$\underline{ABBAb}AABAAbababaabbABBB\underline{Babba}aB$$ and
the joint $aa,AA$ is contained in a maximal joint of length $4$:
$$ABBAb\underline{AABA}Abab\underline{abaa}bbABBBBabbaaB$$
whereas the joint $bb,BB$ of length $2$ is already maximal:
$$ABBAbAABAAbababaa\underline{bb}AB\underline{BB}BabbaaB$$
\end{example}

The next lemma, although a simple consequence of the estimates in \S~\ref{L<1subsection}, 
is key. It shows that with very high probability, the collection
of {\em all} tripods of length $(1/2-\epsilon)m$ can be almost exactly glued up
in pairs:

\begin{lemma}\label{tripod_boundary_small}
Let $L/m=1/2-\epsilon$. Then with probability $1-O(C^{-n^c})$ there is an inequality
$|\partial \mu - \iota \partial \mu| = O(n^{-\epsilon/3} |\mu|)$,
where $\mu$ is the uniform measure on $T_L$, and $|\cdot|$ denotes mass of a (possibly signed)
measure.
\end{lemma}
\begin{proof}
For any given $J(x)$ in $v$ contained in a maximal $J(x')$
we estimate the number of tripods $T(x',y,z)$ with 
$J(x')$ in $\partial T$. First of all, $y$ is determined, since the copy of $x$ associated
to $J$ is the initial subword of some $xY$. Similarly, $z$ is determined, since the
copy of $X$ associated to $J$ is the terminal subword of some $zX$. Therefore
the number of tripods is simply equal to the number of subwords of the form $yZ$ in $v$.

The number of copies of $yZ$ in $v$ is approximately 
$n/|F_{2L}|$, i.e.\/ about $n^{\epsilon}$ with an
error of size $O(n^{\epsilon/2+\delta})$ for any $\delta$, by Proposition~\ref{small_error}.
Taking $\delta = \epsilon/6$ for concreteness, the error is at most $O(n^{2\epsilon/3})$ which
is a fraction $O(n^{-\epsilon/3})$ of the total mass.

Since this is true for every joint $J(x)$, the lemma follows. 
\end{proof}

\subsection{Proof of upper bound}

The proof of Proposition~\ref{upper_bound_proposition}, is now straightforward:
\begin{proof}
Assemble the tripods and glue them in pairs along their common boundary joints.
By Lemma~\ref{tripod_boundary_small} all but $O(n^{-\epsilon/3})$ of the measure of the set
of tripods can be glued up this way, with probability $1-O(C^{-n^c})$. This holds even
conditioning on $v\in [F,F]$ with probability $1-O(C^{-n^c})$, with slightly different constants,
by Theorem~\ref{Sharp_theorem}.

This (partial) fatgraph $Y$ can be extended (usually in many ways) to a
complete fatgraph bounding some multiple of $v$ in $B_1^H(F)$ so that the
Euler characteristic of the added surface is proportional to the mass of the unglued part.
We explain how to do this.

Let $N$ be the function on the letters of $v$ whose value at a given letter is the number of
edges of tripods that contain it, and let $N'$ be the maximum of $N$. The function $N'-N$
is therefore non-negative, and on the other hand $\max N'-N = O(n^{-\epsilon/3})N'$. We translate
the problem of building a fatgraph that extends $Y$ as a problem of suitably 
gluing together a collection of rectangles. 

Each rectangle corresponds to some finite subword $w$ of $v$ which we call the {\em label} of the
rectangle. We think of the rectangle as having
height 1 and width equal to the length of $w$. We keep track not only of $w$ as a word in the 
generators, but also of {\em where} it appears as a subword of $v$. Color the top horizontal
edge of the rectangle blue, and the vertical sides red. 

\begin{figure}[htpb]
\labellist
\small\hair 2pt
\endlabellist
\centering
\includegraphics[scale=0.6]{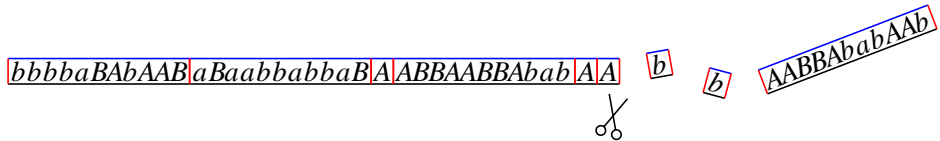}
\caption{$N'$ copies of $v$ cut up into long and short rectangles}\label{rectangle_figure}
\end{figure}

We want to glue together rectangles along
segments of the boundary of integer length, blue to blue and red to red, so that two red edges 
may be glued only if the words associated to the rectangles are consecutive subwords of $v$, and two 
blue segments are glued only if the paired letters on either side are inverse in $F$.

We take three rectangles for each copy of each tripod, with labels the subwords of $v$
corresponding to the edges of the tripod. We also take $N'- N$ rectangles for each letter of $v$,
with label that letter. So we have lots of ``long'' rectangles --- three for each tripod --- and
far fewer ``short'' rectangles (of length 1). By the definition of $N'$ and $N$, every letter
of $v$ appears as the rightmost letter of a label exactly as many times as the following letter of
$v$ appears as the leftmost letter of a label. So we could think of taking $N'$ strips labeled
$v$ and cutting them into long and short rectangles; see figure~\ref{rectangle_figure}.
Naturally, it is possible to glue up the
red segments in pairs compatibly. However, there are potentially many ways to do this, and it is
important to glue up blue edges first, as we now explain.

The long rectangles can be glued up along blue edges
in threes to build fattened tripods. Pairs of tripods can then be glued up along red edges 
corresponding to joints (note that pairs of tripods are glued up in this manner along red
segments of length 2). The result can be thought of in an obvious way as the partial fatgraph $Y$,
where the blue edges are the core graph. See figure~\ref{gluing_tripods_figure}.

\begin{figure}[htpb]
\centering
\includegraphics[scale=0.6]{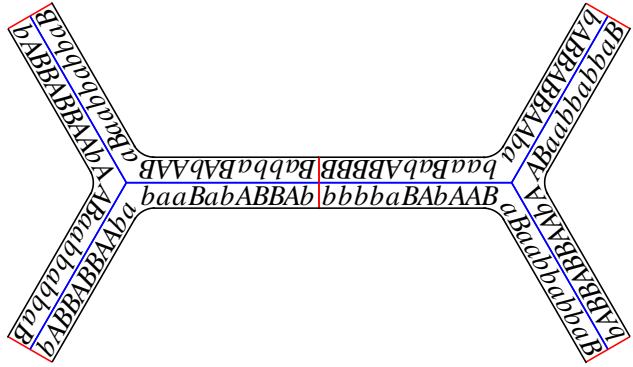}
\caption{Long rectangles glued in threes along blue edges to make tripods, and tripods glued in pairs along red
edges.}\label{gluing_tripods_figure}
\end{figure}

Recall that by hypothesis $v$ is homologically trivial, and note that
the rectangles corresponding to a given tripod have the same number of copies of each generator
as of its inverse. Consequently for each generator of $F$, there are as many short rectangles
labeled with this generator as are labeled with its inverse. We can therefore glue together these
short rectangles in pairs, so that every blue edge can be thus glued up. 

As observed above, the remaining unglued red segments can be glued up in pairs. We now perform
this gluing (in an arbitrary way). See figure~\ref{final_gluing_figure}. Note that the result
might have corners at which more than two paired red edges meet.

\begin{figure}[htpb]
\centering
\includegraphics[scale=0.5]{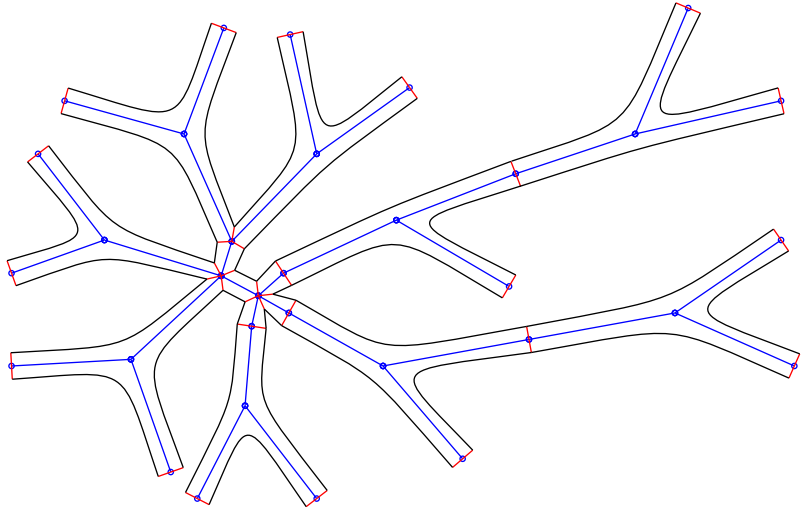}
\caption{Remaining red segments can be glued in pairs. This might produce red corners where more than
two red edges meet.}\label{final_gluing_figure}
\end{figure}

The resulting surface $Y'$ has no unglued red edges. The blue edges
form the core of the surface, and the labels and the way the blue edges sit in the surface
amounts to giving it the structure of a fatgraph over $F$ whose boundary is a multiple of $v$.
The fatgraph $Y$ sits in $Y'$ in an obvious way, and the contribution of $Y'-Y$ to
$-\chi$ is of order $(N'-N)|v|$, which is very small compared to the contribution from $Y$. In
particular, the average edge length $\ell m$ of $Y'$ differs from the
average edge length of $Y$ by at most $O(n^{-\epsilon/3})$, and therefore
satisfies $\ell > 1/2-\epsilon$. The proof now follows from 
Lemma~\ref{combinatorial_estimate}.
\end{proof}

\begin{remark}
The use of ergodic theory to construct an almost
equidistributed collection of pieces with prescribed geometry
that can be almost glued up is inspired by the techniques in Kahn-Markovic's recent proof 
\cite{Kahn_Markovic} of the surface subgroup conjecture in $3$-manifold topology, and
we are pleased to acknowledge our intellectual debt to this paper.
\end{remark}

\subsection{Lower bounds}

The goal of the next few sections is to prove the following estimate, which precisely complements
Proposition~\ref{upper_bound_proposition}. The Random Rigidity Theorem (i.e.\/ 
Theorem~\ref{random_rigidity_theorem}) follows immediately from these two propositions.

\begin{proposition}\label{lower_bound_proposition}
Let $v$ be a random reduced word in the commutator subgroup of length $n$. Then for
any $\epsilon>0$ and any $C$,
$$\log(2k-1)/6 - \scl(v)\log(n)/n \le \epsilon$$ 
with probability $1-O(n^{-C})$.
\end{proposition}

Note that the probability estimate associated to the upper bound is exponential, whereas
the estimate associated to the lower bound is merely polynomial (of arbitrarily large degree).
This disparity is an artifact of the method of proof. A worse lower bound, but with
exponential bounds on the probability of deviation, is obtained in 
\S~\ref{quasimorphism_section} using the method of {\em quasimorphisms}.

\subsection{Combs}

Let $b$ be a subword of $v$, and consider some copy of $b$ in the boundary of an extremal
fatgraph $Y$ for $v$. Recall that by our convention we artificially split open vertices
of higher valence so that $Y$ is trivalent, although it might have some edges of length $0$.
The subword $b$ is contained in a segment $\sigma$ of $Y$, which
is incident to a sequence of edges $e_1,e_2,\cdots, e_d$ of $Y$ in order. Call the subgraph
of $Y$ consisting of the support of $b$ together with the union of the $e_i$ a {\em comb}.

Let $c_1,\cdots, c_d$ be the labels on the edges $e_i$ (oriented to point in to $\sigma$).
Furthermore, the vertices of the $e_i$ subdivide $b$ into subwords $b_0,\cdots, b_d$, where
we stress that some $b_i,c_i$ might have length $0$.
Then there are boundary labels of $Y$ of the form 
$B_d C_d, c_dB_{d-1}C_{d-1}, \cdots, c_2B_1C_1, c_1B_0$ (see Figure~\ref{comb_figure}). 
By the definition of an extremal fatgraph, these boundary labels are (cyclic) subwords of $v$.

\begin{figure}[htpb]
\labellist
\small\hair 2pt
\pinlabel $b_0$ at 25 -2
\pinlabel $b_1$ at 65 -2
\pinlabel $b_2$ at 105 -2
\pinlabel $b_3$ at 145 -2
\pinlabel $b_4$ at 185 -2
\pinlabel $B_0$ at 25 12
\pinlabel $B_1$ at 65 12
\pinlabel $B_2$ at 105 12
\pinlabel $B_3$ at 145 12
\pinlabel $B_4$ at 185 12
\pinlabel $c_1$ at 38 25
\pinlabel $c_2$ at 78 25
\pinlabel $c_3$ at 118 25
\pinlabel $c_4$ at 158 25
\pinlabel $C_1$ at 53 25
\pinlabel $C_2$ at 93 25
\pinlabel $C_3$ at 133 25
\pinlabel $C_4$ at 173 25
\endlabellist
\centering
\includegraphics[scale=1]{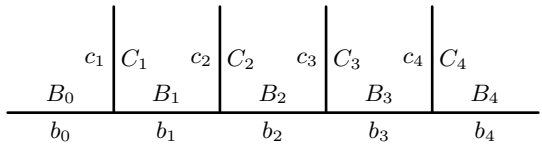}
\caption{A comb with edge labels.}\label{comb_figure}
\end{figure}

This suggests the following definition:

\begin{definition}
Given a word $b \subset v$ a {\em comb on $b$} 
is a family of subwords of $v$ of the
form $B_dC_d, c_dB_{d-1}C_{d-1}, \cdots , c_1B_0$.
The {\em complexity} of the cone is $d$ (as above) and the {\em length} is
$L$, where $|b| + \sum |c_i| = Lm$.
\end{definition}

We would like to bound (in probability) the length of a comb in terms of its
complexity. Fix a big constant $L'$, and let $b\subset v$ be a subword of length $L'm$.
We would like to construct a comb on $b$ for which $L/d$ is as big as possible.
This amounts to choosing a partition of $b$ into $d$ successive subwords $b_i$ of
length $L_im$ (where $L_i=0$ is allowed), then choosing copies of $B_i$ in $v$, and defining
$c_i$ to be the maximal subword following the copy of $B_i$ for which $C_i$ precedes 
the copy of $B_{i-1}$. Let these maximal $c_i$ have length $K_im$.

Note that the comb has length $L=\sum_{i=0}^d L_i + \sum_{i=1}^d K_i$ and complexity $d$.
We would like to bound in probability the maximum ratio $L/(2d+1)$, at least for
typical $b$ of some fixed length $L'm$ where $L'=\sum L_i$.

\medskip

By Proposition~\ref{small_error} and Proposition~\ref{long_inverse_word_bound}
there are almost exactly $n^{1-L_i}$ possible locations of each $B_i$ in $v$ for $L_i<1$,
and the chance that there is some $B_i$ at all when $L_i>1$ is at most $n^{1-L_i}$. If we assume
that the prefixes and suffixes of the $B_i$ of fixed length are evenly distributed, 
then for any fixed $T$, there should be an estimate
$$\Pr\left(\sum K_i > T+\sum(1-L_i)\right) = O(n^{-T})$$
If $T$ is very big but fixed, and small compared to $L'=\sum L_i$, then we can estimate
$L \le T+(d+1)$, and therefore $L/(2d+1) \le 1/2+\epsilon$ for any
$\epsilon$ with probability $1-O(n^{-T})$. This is good enough to give the desired bound in 
Proposition~\ref{lower_bound_proposition}, by Lemma~\ref{combinatorial_estimate}.

Notice that this heuristic argument is almost rigorous: prefixes and suffixes of the $B_i$ are
not perfectly independent, but their correlation decays exponentially fast with the
distance between $B_i$ and $B_j$. Thus we need only examine the cases in which there
are $C_{i+1}B_ic_i$ and $C_{j+1}B_jc_j$ that overlap. In order to obtain the desired
estimate, it is necessary to make some {\it a priori} assumptions about a cone on $b$,
which will turn out to be justified for {\em most} combs in any given extremal fatgraph $Y$.

\begin{definition}
A subword $b$ of $v$ is {\em $\delta$-regular} if there is no subword $b'$ of length 
$(1+\delta)m$ such that $B'$ is in $v$, and if all subwords of $b$ of 
length $\delta$ and their inverses are distinct.

A comb on $b$ is {\em $\delta$-regular} if $b$ is $\delta$-regular, and if all the
$c_i$ have length at most $(1+\delta)m$.
\end{definition}

Let $v$ be a random word in $F'$ of length $n$, and let $Y$ be an extremal trivalent 
fatgraph for $v$ (possibly with some edges of length $0$). For any $d$, we can consider
the set of combs of $Y$ of complexity $d$. The following lemma justifies the
definition of $\delta$-regular:

\begin{lemma}\label{delta_regular_lemma}
Let $v$ be a random word in $F'$ of length $n$, and let $Y$ be an extremal fatgraph for $v$.
Then for any $d$, the proportion of combs of $Y$ of complexity $d$ that are not $\delta$-regular
is at most $O(n^{-\delta/2})$, with probability $1-O(C^{-n^c})$.
\end{lemma}
\begin{proof}
By Proposition~\ref{long_inverse_word_bound}, with probability $1-O(C^{-n^c})$ there are
at most $n^{1-\delta/2}$ subwords of $v$ of length $(1+\delta)m$ whose inverse also
appears in $v$ (in fact, we could take any number $<\delta$ in place of $\delta/2$);
hence the proportion of combs of complexity $d$ that contain an edge of length
$\ge (1+\delta)m$ is at most $(4d+2)n^{-\delta/2}$, since every edge of $Y$ is
contained in $4d+2$ combs of complexity $d$, and $\partial Y$ represents $Nv$ for some $N$.

An argument similar to Proposition~\ref{large_cardinality} 
establishes that subwords of typical $b$ of length $\delta$ are 
distinct, with probability $O(C^{-n^c})$.
\end{proof}

\subsection{Overlaps}

We now restrict attention to a fixed $\delta$-regular word $b$, and
consider a random word $v$ conditioned to contain $b$ as a subword.
The arguments in this section depend on order-of-magnitude estimates
of probability, expressed as a power of $n$.

Fix vectors of lengths $L_i, K_i <(1+\delta)m$, 
and for each choice of $d$
locations in $v$, consider the probability that the subwords $e_iD_{i-1}C_i$
of length $K_i+L_{i-1}+K_{i-1}$ starting at these locations constitute a comb on
$b$; we call such an occurrence a {\em matching}, and we want to estimate the
probability of a matching at a given $d$-tuple of locations. 
We also refer to a vector of $d$ locations in $v$ as above as a {\em configuration}.
If the subwords do not overlap, this probability is less than
$n^{-(\sum L_i + \sum K_i)}$. So it suffices to estimate the probability
in the case that some subwords do overlap. This is somewhat fiddly,
and depends on an analysis of the combinatorial possibilities for the overlap.
However, the estimates in every case are entirely elementary.

For each $j\ge 2$, let $P_jm$ be the total length where at least $j$ words
overlap. Define the {\em total overlap}, counted with multiplicity, to be 
$P: = \sum_{j\ge 2} P_j$. The total contribution to $P$ from overlaps of $D_i$
with $D_j$ will be $O(\delta)$, since $b$ is $\delta$-regular.
If part of some $e_i$ (resp. $C_i$) is contained in an overlap, but the corresponding
part of $C_i$ (resp. $e_i$) is not, this overlap does not significantly
affect the probability of a matching. If corresponding parts of $C_i, e_i$
both overlap $D_j$, then again necessarily this overlap will be of size
$O(\delta)m$, since $b$ is $\delta$-regular. So to estimate the probability of
a matching, it suffices to consider overlaps among the various $C_i,e_j$.
Let $P'_jm$ be the total length where at least $j$ such subwords overlap,
and analogously define $P':=\sum_{j\ge 2} P'_j$. 

\begin{lemma}
With notation as above, the probability of a matching in a given configuration 
is at most $n^{P'/2-(\sum L_i + \sum K_i)+O(\delta)}$.
\end{lemma}
\begin{proof}
An overlap in some subword of $e_i$ of length $lm$ 
must correspond to an overlap in the
corresponding subword of $C_i$ to increase the probability of a match by at
most $n^l$; so the increase over the ``naive'' probability of a match is at
most a factor of $n^{P'/2}$.
\end{proof}

On the other hand, there are $n^d$ sets of locations of the subwords, and for
each given location of one subword, there are only $O(\log(n))$ locations of
any other subword that overlaps it. Two subwords $e_iD_{i-1}C_{i-1}$ and
$e_jD_{j-1}C_{j-1}$ can contribute at most $2(1+\delta)$ to $P'$, precisely
if $e_i=e_j$ and $C_{i-1}=C_{j-1}$. We deduce the following lemma:

\begin{lemma}\label{comb_estimate}
Let $L_i, K_i$ be some fixed vector of lengths with $L_i,K_i<1+\delta$, and define
$L=\sum_i L_i + \sum_i K_i$. Suppose $b$ is an $\delta$-regular subword of $v$.
Then the probability that there is a comb over $b$ with the prescribed 
lengths is at most $O(n^{-T+O(\delta)})$ where $T=L-d-1$. Consequently if
$L/(2d+1) \ge 1/2+\epsilon$ and $\delta$ is sufficiently small compared to $\epsilon$,
and $d$ is sufficiently big compared to $\epsilon$, we can make $T$ as big as desired.
\end{lemma}
\begin{proof}
As above, each set of locations has probability at most $n^{P'/2 - L + O(\delta)}$ of a
matching. Moreover, there are $n^d$ sets of locations, and at most $n^{d-r+O(\delta)}$
sets of locations for which $P'\ge 2r$. The estimate follows.
\end{proof}

\subsection{Proof of lower bound}

We now give the proof of Proposition~\ref{lower_bound_proposition}
\begin{proof}
By Lemma~\ref{combinatorial_estimate}, it suffices to show for every $C$ and
every $\epsilon$ that the average length
$\ell m$ of the edges of an extremal fatgraph $Y$ is at most $1/2+\epsilon$, 
with probability $1-O(n^{-C})$. By Theorem~\ref{Sharp_theorem}, conditioning
that $v$ lies in $[F,F]$ only affects probabilities by at most a factor of $O(n^{k/2})$.

By Proposition~\ref{long_inverse_word_bound}, there are only $O(1)$ subwords of 
$v$ of length $\ge 2m$ and $O(n^{1-\delta/2})$
of length $\ge (1+\delta)m$, whose inverse also appears in $v$,
 with probability $1-O(C^{-n^c})$. So edges of length $\ge (1+\delta)m$ affect
$\ell$ negligibly, and the fraction of combs containing such subwords are
similarly negligible.

Choose some very large constant $d$, roughly of size $O(1/\epsilon)$, 
and consider the set of all combs with complexity
$d$ in $Y$. Because $Y$ is (formally) trivalent, every edge occurs in exactly $(4d+2)$
such combs --- each comb has $2d+1$ edges, and each edge has two sides.
By Lemma~\ref{delta_regular_lemma}, if $\ell \ge 1/2+\epsilon$, a definite 
fraction of these combs must be $\delta$-regular, and satisfy $L/(2d+1) > 1/2+\epsilon$. 

On the other hand, by Lemma~\ref{comb_estimate}, for any $\delta$-regular subword $b$
and any given vector of lengths $<1+\delta$ the probability that there is a comb over
$b$ with prescribed lengths is at most $O(n^{-T+O(\delta)})$ where $T=L-d-1$.
Since there are at most $n$ possible locations in $v$ for such a subword $b$, and
since there are at most $((1+\delta)m)^{2d+1} < n^\delta$ vectors of lengths, the
probability that there is {\em any} $\delta$-regular comb with complexity $d$ and
length $L$ is at most $O(n^{-T+1+O(\delta)})$. So for any $C$ and any $\epsilon$,
if $d$ is sufficiently large and $L/(2d+1) > 1/2 + \epsilon$, 
{\em no} such comb exists, with probability $1-O(n^{-C})$. The proof follows.
\end{proof}

\begin{remark}
A more careful analysis would almost certainly improve the estimate of the probability
of a large negative deviation. The probability that a specific $\delta$-regular subword
is part of a $\delta$-regular comb with big $d$ and $L/(2d+1)>1/2+\epsilon$ 
is polynomial in $n$, and to violate the desired
lower bound on $\scl$ we must construct a fatgraph containing a {\em definite proportion}
of such big $\delta$-regular combs. However, the events that distinct subwords
$b,b'$ are parts of such $\delta$-regular combs are not obviously independent, and even
estimating their correlation appears hard. Nevertheless, heuristically one would
expect the true probability of a deviation to be exponential in (some power of) $n$.
\end{remark}

\subsection{The Random Norm Theorem}

In fact, it is not much more work to derive the following theorem, which specializes
to Theorem~\ref{random_rigidity_theorem} when $d=1$:

\begin{theorem}[Random Norm Theorem]\label{random_norm_theorem}
Let $F$ be a free group of rank $k$, and for fixed $d$, let $v_1,v_2,\cdots,v_d$ be 
independent random
reduced elements of length $n_1,n_2,\cdots,n_d$ conditioned to lie in $[F,F]$, where
without loss of generality we assume $n_1\ge n_i$ for all $i$. Let $V$ be the subspace of
$B_1^H(F)$ spanned by the $v_i$. Then for any $\epsilon>0, C>1$ and real numbers $t_i$,
$$|\scl(\sum t_iv_i)\log(n_1)/n_1 - \log(2k-1)(\sum |t_i|n_i)/6n_1|\le \epsilon$$
with probability $1-O(n_1^{-C})$.
\end{theorem}

We remark before giving the proof
that even though the $C^0$ geometry of a (random) slice of the unit ball is very
simple, the finer polyhedral structure is apparently extremely complicated. Figure~\ref{AbaBAbaBBAbaabaaBAAA_baaabAAbaabABAABBABa}
and Figure~\ref{simple} exhibit $2$ and $3$ dimensional slices of the $\scl$ unit ball of
some relatively simple words.

\begin{figure}[htpb]
\labellist
\small\hair 2pt
\endlabellist
\centering
\includegraphics[scale=0.25]{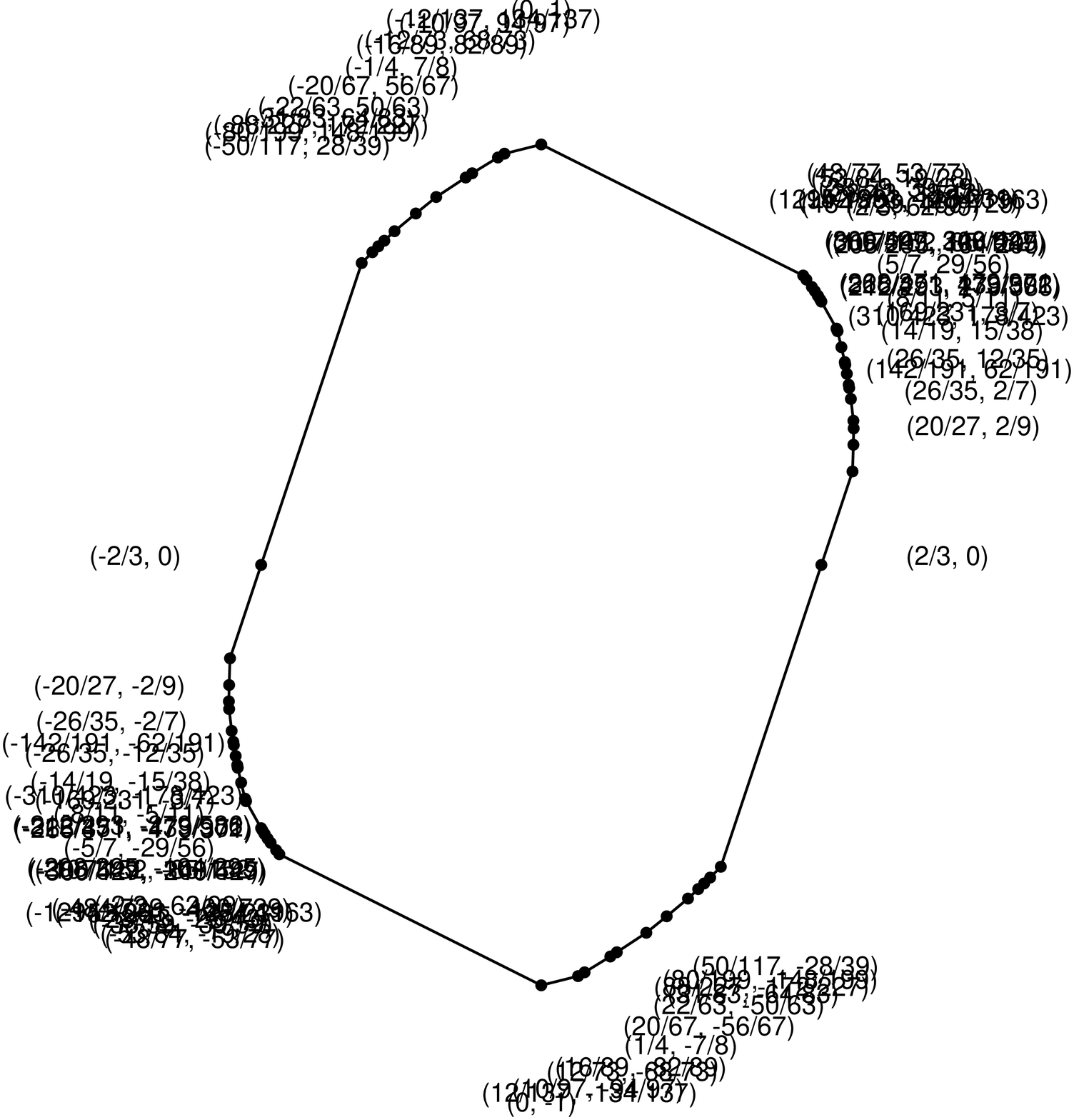}
\caption{Unit ball in the $\scl$ norm in the subspace spanned by $AbaBAbaBBAbaabaaBAAA$ and
$baaabAAbaabABAABBABa$}\label{AbaBAbaBBAbaabaaBAAA_baaabAAbaabABAABBABa}
\end{figure}

\begin{figure}[htpb]
\labellist
\small\hair 2pt
\endlabellist
\centering
\includegraphics[scale=1]{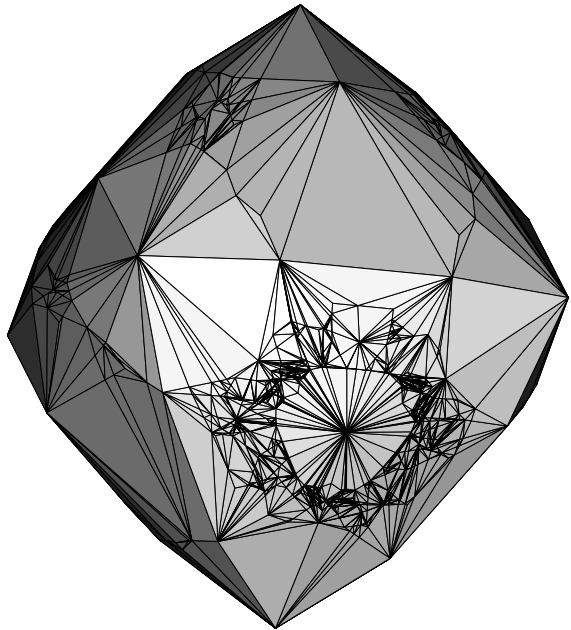}
\caption{Unit ball in the $\scl$ norm in the subspace spanned by $aabAcAcBCC$,
$bbcBaBaCAA$ and $ccaCbCbABB$}\label{simple}
\end{figure}

\begin{proof}
We give the proof in the case $d=2$; the general case follows by essentially the same argument.
For any pair of reduced words $v_1,v_2$ (not necessarily in $[F,F]$)
choose a word $z$ of length $6$ contained in $[F,F]$ 
so that $v_1zv_2$ is reduced. We can always find such a word $z$ of the form $xyXXYx$ for some generators 
$x,y$ so that $v_1$ does not end, and $v_2$ does not begin, with $X$.

This defines a map $F_{|v_1|}\times F_{|v_2|} \to F_{|v_1|+|v_2|+6}$, and the
pushforward of the product of uniform measures is proportional to the uniform measure on
the image, with constant of proportionality independent of $n$. The relative proportion
of the image is a constant, so by Theorem~\ref{random_rigidity_theorem} for any $\epsilon>0, C>1$
we have
$$|\scl(v_1zv_2) \log(n_1+n_2+6)/(n_1+n_2+6) - \log(2k-1)/6|\le \epsilon$$ 
with probability $1-O(n_1^{-C})$.
For $n_1\ge n_2$ large, $\log(n_1+n_2+6)$ is very close to $\log(n_1)$. On the other hand,
$|\scl(v_1+v_2) - \scl(v_1zv_2)| \le \text{const}$. It follows for any $\epsilon>0,C>1$,
with probability $1-O(n_1^{-C})$,
$$|\scl(v_1+v_2) - \scl(v_1) - \scl(v_2)| \le \epsilon n/\log(n)$$
In particular, the boundary of the unit ball contains a point which is very close to the
midpoint of the points $v_1/\scl(v_1)$ and $v_2/\scl(v_2)$, and by convexity, the unit
ball in the positive quadrant of the $v_1,v_2$ plane is $C^0$ close to a triangle.
Replacing $v_i$ by $v_i^{-1}$, the entire unit ball in the $v_1,v_2$ plane is $C^0$ close to
a diamond. The higher dimensional case is completely analogous.
\end{proof}

\section{Quasimorphism lower bound}\label{quasimorphism_section}

In this section we exhibit an explicit quasimorphism which certifies a uniform lower 
bound for $\scl$ of a random word.  Unfortunately, this lower bound is not sharp, 
for it exhibits only $\scl(v) \ge n\log(2k-1)/12\log(n)$ (with high 
probability), which is $1/2$ of the correct value, by Theorem~\ref{random_rigidity_theorem}.

Experience shows that constructing explicit extremal quasimorphisms is difficult. For
example, there is a polynomial time algorithm to produce an extremal surface 
for a chain in a free group, whereas 
there is no known algorithm (of any kind) to produce a certifying quasimorphism. 
Bj\"orklund-Hartnick \cite{Bjorklund_Hartnick} proved a central limit theorem for
quasimorphisms (on random walks; but these are very similar to random words in the
special case of free groups), and consequently any fixed quasimorphism on $F$ 
takes values of order $O(\sqrt{n})$
on words of length $n$. For this reason, it is interesting to be able to construct
an explicit quasimorphism which gives the correct $O(n/\log(n))$ order of magnitude.
Another nice feature of the construction is that the bound in probability is exponential in
$n$, in contrast to the polynomial bound in Proposition~\ref{lower_bound_proposition}.

\subsection{Quasimorphisms and Bavard Duality}

A reference for the material in this section is \cite{Calegari_scl}, especially
Chapter~2.

\begin{definition}
Let $G$ be a group a {\em quasimorphism} is a function for which there is
a least non-negative real number $D(\phi)$ (called the {\em defect}) for which
$$|\phi(gh)-\phi(g)-\phi(h)|\le D(\phi)$$
for all $g,h \in G$.

Furthermore, a quasimorphism is {\em homogeneous} if $\phi(g^n)=n\phi(g)$ for all
$g\in G$ and all integers $n$.
\end{definition}

If $\phi$ is any quasimorphism, the {\em homogenization} of $\phi$, denoted
$\overline{\phi}$, is defined by
$$\overline{\phi}(g):=\lim_{n \to \infty} \phi(g^n)/n$$
It is a fact that $\overline{\phi}$ is a homogeneous quasimorphism, and satisfies
$D(\overline{\phi})\le 2D(\phi)$. See \cite{Calegari_scl}, Lemma~2.58.
The set of homogeneous quasimorphisms on $G$ is a real vector space $Q(G)$. The
subspace with $D=0$ consists precisely of the homomorphisms $H^1(G;\R)$, and $D$
makes the quotient $Q/H^1$ into a {\em Banach space}.

There is a duality between quasimorphisms and stable commutator length, known as
{\em Generalized Bavard Duality}. The statement of this duality theorem is:
\begin{theorem}[Generalized Bavard Duality \cite{Calegari_scl}, Thm.~2.79]
Let $G$ be a group. Then for any $\sum t_i g_i \in B_1^H(G)$ there is an equality
$$\scl(\sum_i t_ig_i) = \frac 1 2 \sup_{\phi \in Q/H^1} \frac {\sum_i t_i \phi(g_i)} {D(\phi)}$$
\end{theorem}
A special case of this theorem was established by Bavard in \cite{Bavard}.
Notice that this theorem is ``complementary'' to Lemma~\ref{surface_lemma}: an admissible
surface certifies an upper bound for $\scl$, whereas a homogeneous quasimorphism certifies
a lower bound.

An important and useful class of quasimorphisms are the (big) counting quasimorphisms, defined
by Rhemtulla \cite{Rhemtulla}, and rediscovered by Brooks \cite{Brooks}.
Recall the definition of the counting functions $C_\sigma$ from \S~\ref{counting_function_subsection}.
and their antisymmetrization $H_\sigma:=C_\sigma-C_{\sigma^{-1}}$. 
Given a set of reduced words $S\subset F$, 
the function $H_S:=\sum_{\sigma \in S} H_\sigma$ is a quasimorphism, and its value on
$v$ counts the difference in the number of copies of $\sigma$ and of $\sigma^{-1}$ for
each $\sigma \in S$. The homogenization counts the difference of the number of copies in the
(cyclically reduced) {\em cyclic} word $v$.

While big counting quasimorphisms are intuitively very natural, it will be technically easier 
for us to work with {\em small} counting quasimorphisms.  As above, let $S \subset F$, and define 
\[
 c_S = \text{maximal number of disjoint copies of elements of }S\text{ in }v.
\]
Then $h_S = c_S - c_{S^{-1}}$ is a quasimorphism, the small counting quasimorphism on $S$.  See 
e.g.~\cite{Calegari_scl}~\S~2.3.2. 
In contrast to big counting quasimorphisms, for which bounding the defect proves difficult, 
small counting quasimorphisms have a uniformly bounded defect. 
\begin{lemma}
\label{small_counting_defect}
For any $S \subseteq F$, we have $D(h_S) \le 3$ and $D(\overline{h}_S)\le 6$.
\end{lemma}
\begin{proof}
This is Lemma~5.1 from \cite{Calegari_isometry}.
\end{proof}

\subsection{Construction of the quasimorphism}

\begin{proposition}\label{quasimorphism_lower_bound}
Let $v$ be a random reduced word in the commutator subgroup of length $n$. Then
there is an explicit construction of a homogeneous quasimorphism, so that 
for all $\epsilon>0$ there are constants $C>1$ and $c>0$ such that with 
probability $1-O(C^{-n^c})$, the quasimorphism certifies the inequality
$$\scl(v) \ge \frac{1}{1+\epsilon} \frac{n\log(2k-1)}{12\log(n)}$$
\end{proposition}
\begin{proof}
Recall our notation $m=\log(n)/\log(2k-1)$ where $k$ is the rank of the free group $F$.
Fix $L=1+\epsilon$ for $\epsilon>0$, and partition the cyclic word $v$ into adjacent 
disjoint subwords of length $Lm$.  Note that there may be some small remainder if $Lm$ 
does not divide $n$; ignore this gap, as it will be insignificant for our purposes.  
Let $S$ be the collection of these subwords. 
\begin{lemma}\label{quasi_set_construction_lemma}
For $L=1+\epsilon$ and $S$ as above, there exist $C>1$ and $c>0$ such that with 
probability $1-O(C^{-n^c})$, there is a subset $S'\subset S$ with 
\[
 \text{card}(S-S') < n^{2-L+\epsilon/2}
\]
such that for no $\sigma \in S'$ does $\sigma^{-1}$ appear in $v$.
\end{lemma}
\begin{proof}
Repeating the content of \S~\ref{L>1subsection} while assuming that the words in 
$S$ are disjoint only simplifies the arguments, so Proposition~\ref{long_inverse_word_bound} 
still holds in this case.
\end{proof}

The certifying quasimorphism will be $\overline{h}_{S'}$.  By construction, 
\[
 \overline{h}_{S'}(v) \ge \frac{n}{Lm} - n^{2-L+\epsilon/2} - 1 - c_{(S')^{-1}}(v),
\]
and $c_{(S')^{-1}}(v) = 0$ by Lemma~\ref{quasi_set_construction_lemma}.  
By Lemma~\ref{small_counting_defect}, $D(\overline{h}_{S'}) \le 6$, so Bavard duality gives 
\[
\scl(v) \ge \frac{\overline{h}_{S'}(v)}{2D(\overline{h}_{S'})} 
        \ge \frac{1}{1+\epsilon} \frac{n\log(2k-1)}{12\log(n)} - o(n/\log(n)).
\]
The statement of the lemma is obtained by repeating the argument with $\epsilon/2$; the 
multiplicative factor $1/(1+\epsilon)$ then renders the $o(n/\log(n))$ unnecessary.
\end{proof}

\section{Computer experiments and a surprisingly good heuristic}\label{heuristic_section}

Recall that in the proof of Proposition~\ref{upper_bound_proposition} we constructed a surface 
by gluing random tripods. The length of the edges of the tripods was $m/2 = \log(n)/2\log(2k-1)$, 
but each edge of each tripod was extended to a maximal joint before gluing. 
If $u=xy$ and $u'=xy'$ are reduced words with a common nonempty
prefix $x$, the expected length of the common prefix of $y$ is
$1/(2k-1) + 1/(2k-1)^2 + \cdots = 1/(2k-2)$. This suggests that the average edge length
of an extremal surface should be at least $m/2 + 1/(2k-2)$, and therefore that the
value of $\scl$ should be at most $n/12 (\log(n)/2\log(2k-1) + 1/(2k-2))^{-1}$.

Without a really sound theoretical justification, we nevertheless made the prediction that
this heuristic correction should more accurately match the actual average value of $\scl$,
and tested this experimentally.

Figure~\ref{experiment_figure} displays the result of computer experiment. We computed
the $\scl$ of $20$ random words in $[F_2,F_2]$ of lengths  
between $70$ and $240$ (inclusive) in steps of $10$. The upper solid line indicates the
theoretical value $n\log(2k-1)/6\log(n)$ from Theorem~\ref{random_rigidity_theorem}, 
the dots are the actual averages, and the lower dashed line (passing in a very satisfying way
through the experimental dots!) is the heuristic $n/12 (\log(n)/2\log(2k-1) + 1/(2k-2))^{-1}$.

\begin{figure}[htpb]
\labellist
\small\hair 2pt
\pinlabel $70$ at 140 0
\pinlabel $100$ at 200 0
\pinlabel $130$ at 260 0
\pinlabel $160$ at 320 0
\pinlabel $190$ at 380 0
\pinlabel $220$ at 440 0
\pinlabel $1$ at 120 20
\pinlabel $3$ at 120 60
\pinlabel $5$ at 120 100
\pinlabel $7$ at 120 140
\pinlabel $\scl$ at 100 80
\pinlabel $\text{word length}$ at 300 -20
\endlabellist
\centering
\includegraphics[scale=0.75]{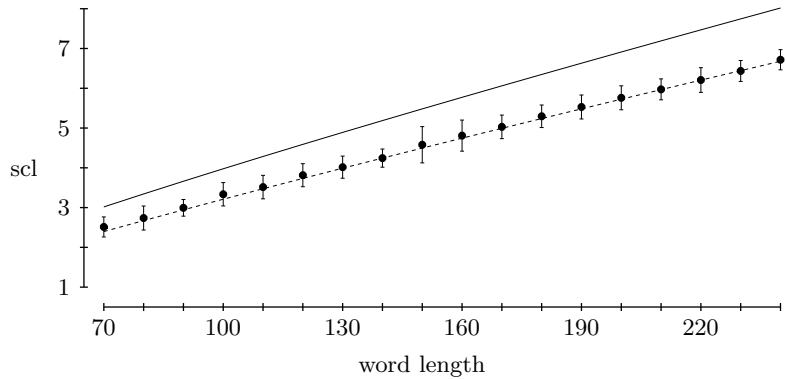}
\caption{Experimental computation of $\scl$ on random words in $F_2$
of length between $70$ and $240$, and comparison with
(asymptotic) theoretical and heuristic values.}\label{experiment_figure}
\end{figure}

\appendix

\section{Directed graphs and Markov chains}\label{Markov_chain_section}

The purpose of this appendix is firstly to put the estimates obtained in 
\S~\ref{random_word_section} into the more general context of the theory of
nonreversible Markov chains, and secondly to indicate which
aspects of the theory developed above can be expected to generalize easily to 
hyperbolic groups and spaces, and which aspects require new ideas.
The main results of the paper do not depend logically
on the results or conjectures in this appendix.

\bigskip

Let $X_1$ be the directed graph whose vertices are the generators of $F$, and whose
(directed) edges are the (ordered) non-inverse pairs.
A random word $v$ of length $n$ can be interpreted as a random walk on $X$ (where
edges have the uniform probability) starting at a random vertex (also with the
uniform probability). This graph is {\em ergodic} (i.e.\/ there is a directed path
from any vertex to any other vertex) and {\em aperiodic} 
(i.e.\/ the gcd of the lengths of the directed loops is $1$).

For any $i$ let $X_i$ be the directed graph whose vertices are the elements of
$F_i$, and whose (directed) edges are the elements of $F_{i+1}$, where an
edge $g$ starts/ends at its prefix/suffix respectively of length $i$.
Note for each $i$ that $X_i$ is $(2k-1)$-regular, ergodic and aperiodic. 
Again, a random word $v$ of length $n$ can be interpreted as a random walk on $X$
of length $n-i+1$ starting at a random vertex.

Each $X_i$ determines a nonreversible Markov chain (in the obvious way), with 
stationary probability $\pi$ the uniform probability measure on vertices (i.e.\/ such
that each vertex has weight $1/(2k)(2k-1)^{i-1}$), and Markov kernel $P_i(x,y)=1/(2k-1)$ if
there is a directed edge from $x$ to $y$; i.e.\/ if $x$ and $y$ are reduced words of length $i$, and
the suffix of $x$ of length $i-1$ is equal to the prefix of $y$ of length $i-1$.

For an introduction to the theory of Markov chains, see \cite{Freedman}. We remark
that we use only the most elementary aspects of the theory in this paper, since our
Markov chains always have discrete time and finite state space.

\subsection{Chernoff inequalities for nonreversible Markov chains}

We would like to estimate the rate of
convergence of random sums to the equilibrium; that is, we want to estimate the
probability that $|n^{-1}\sum_{j=1}^n f(x_j) - \int fd\pi|$ is bigger than $n^{-\delta}$,
for some function $f$ on the vertices of $X_i$ (i.e.\/ on $F_i$). In the sequel
we denote $\int fd\pi$ by $\mean_\pi(f)$, or just $\mean(f)$ if $\pi$ is understood.

As is well-known, for {\em reversible} Markov chains, the rate of convergence is governed
by the spectral gap (i.e.\/ the difference between $1$ and the second largest eigenvalue)
of the (symmetric) Markov kernel $P$. For {\em nonreversible} Markov chains, the relevant
quantity is the smallest nonzero eigenvalue $\lambda_1$ of $L:=\Id -(P+P^*)/2$.
In general $P^*$ is defined by $P^*(x,y)=\pi(y)P(y,x)/\pi(x)$, so in our context $P^*$ is
just the transpose $P^T$.

Let $f$ be normalized to have $\|f-\mean(f)\|_\infty\le 1$ and $\|f-\mean(f)\|_2^2 \le 1$.
Let $q$ be an initial distribution, and define $N_q=\|q/\pi\|_2$ (note that we always have
$N_q \le \min(\pi(v))^{-1/2}$).
Then the main Chernoff-type inequality, due to Lezaud, is as follows:

\begin{theorem}[Lezaud \cite{Lezaud}, Thm.~1.1 (cf. Rmk.~1.3)]\label{Lezaud_estimate}
With notation as above, there is an inequality
$$\Pr\Bigl( n^{-1}\sum_{j=1}^n f(x_j) - \mean(f) \ge \gamma\Bigr) \le N_q e^{-\lambda_1 n \gamma^2/8}$$
\end{theorem}

\begin{remark}
Replacing $f$ by $-f$ gives the same bound on 
$\Pr(n^{-1}\sum_{j=1}^n f(x_j) -\mean(f) \le -\gamma)$.
\end{remark}

\begin{remark}
It is possible to control the rate of convergence in terms of other kinds of
spectral data, for instance, the second smallest eigenvalue $\lambda_1$ 
of $\Id - PP^*$. However for the Markov chains $X_i$ as above with $i\ge 2$, 
the multiplicative reversibilization $PP^*$ has many distinct eigenvectors 
of eigenvalue $1$, so $\lambda_1=0$. Another approach is to work directly with the
smallest positive {\em singular value} of the (nonsymmetric) matrix $\Id -P$; this
approach is favored by Dinwoodie \cite{Dinwoodie}.
\end{remark}

\begin{remark}
Lezaud's estimate is not in itself strong enough to derive Proposition~\ref{small_error}
because the variance of a counting function $C_{\sigma}$ is too big. Nevertheless,
our proof of Proposition~\ref{small_error} owes something to the approach of Lezaud, 
and also to the earlier work of Dinwoodie \cite{Dinwoodie} mentioned above (especially
the implicit estimate of the random covering time in Lemma~\ref{decay_of_correlation}).
\end{remark}

\subsection{Estimating $\lambda_1$}

The following estimate on $\lambda_1$ in terms of the spectrum of $P$ is obtained by
Chung:
\begin{theorem}[Chung \cite{Chung}, Thm.~4.3]\label{Chung_estimate}
If $X$ is a directed graph, the eigenvalue $\lambda_1$ of $L$ is related to the
(ordered) eigenvalues $\rho_i$ of $P$ as follows:
$$\min_{i\ne 0} (1-|\rho_i|) \le \lambda_1 \le \min_{i\ne 0} (1-\text{Re}(\rho_i))$$
\end{theorem}

\begin{remark}
Note that Chung proves her theorem for arbitrary (not necessarily regular) graphs, in
which case the Laplacian $L$ has the more complicated form
$$L = \Id - \frac {\Phi^{1/2}P\Phi^{-1/2} + \Phi^{-1/2}P^*\Phi^{1/2}} 2$$
where $\Phi$ is the diagonal matrix whose entries are the values of $\pi$. For a regular
graph, $\Phi$ is a scalar multiple of the identity and $P^*=P^T$, so 
this simplifies to $\Id - (P+P^T)/2$ which agrees with the definition of 
$L$ in Theorem~\ref{Lezaud_estimate}.
\end{remark}

\begin{lemma}\label{lambda_1_uniformly_bounded}
For $L=\Id-(P_i+P_i^T)/2$ where $P_i$ is the probability matrix for $X_i$, there is
an estimate $\lambda_1 \ge \text{const.}>0$ where $\text{const.}$ does {\em not} depend on $i$.
\end{lemma}
\begin{proof}
By Theorem~\ref{Chung_estimate}, it suffices to obtain upper bounds on the absolute values
$|\rho_i|$ of the spectrum of $P_i$. But the spectrum of $P_i$ is equal to the spectrum of $P_1$
for any $i$ (padded by zeros), since the traces of all powers $P_i^j$ and $P_1^j$ are equal.
To see this, observe that these traces count the number of periodic cycles in $X_i$ and $X_1$
of period $j$, but such cycles in either case are in bijection with bi-infinite periodic
words with period $j$.

So it suffices to show that the spectrum of $P_1$ has a unique eigenvalue $1$ and all other eigenvalues
strictly less than $1$ in absolute value. This follows from the aperiodicity and ergodicity of $X_1$.
\end{proof}

Incidentally, $X_1$ {\em is} a reversible Markov chain, and therefore the spectrum of $P_1$ is
real, so the same is true for the spectrum of all $P_i$.

\subsection{Cheeger constants in $X_i$}

There are other methods to estimate $\lambda_1$ for a directed graph, via a
generalization of the classical {\em Cheeger's inequality}. If $X$ is a regular
directed graph, the {\em Cheeger constant} $h(X)$ is the infimum of $|\partial U|/|U|$ over
all subsets $U$ of vertices of $X$ with cardinality at most $|U|\le|X|/2$, where $\partial U$
is the set of elements of the complement $U^c$ joined by a directed edge from $U$ to $U^c$.

The significance of this quantity for $\lambda_1$ is the following theorem of Chung:

\begin{theorem}[Chung \cite{Chung}, Thm.~5.1]
Let $X$ be a directed graph. Then
$$2h(X)\ge \lambda_1 \ge h^2(X)/2$$
\end{theorem}

For the sake of interest, we show that the Cheeger constants of the $X_i$ are all equal,
which gives another proof of Lemma~\ref{lambda_1_uniformly_bounded}.

\begin{lemma}
For any $i$, there is an equality $h(X_i)\ge h(X_1)$.
\end{lemma}
\begin{proof}
We give a sketch of a proof.

Given $U$ a subset of $X_i$ with $|U|\le |X_i|/2$, let $V$ denote the set of suffixes of $U$ of length
$i-1$, and let $V'$ denote the set of words obtained from $V$ by appending a letter. Then
$\partial U = V'\backslash U$.
Also, let ${}'V$ denote the set of words obtained from $V$ by prepending a letter.
Then $|V'| =|{}'V| = |V|(2k-1)$ and $U\subset {}'V$.
Choose $U$ so that 
$$|\partial U| =|V'\backslash U|= h(X_i)|U|\le h(X_i)|V'|$$
Note that either $|V|\le |X_{i-1}|/2$, or else we may obtain a lower bound on $h(X_i)$ from
the difference $|V| - |X_{i-1}|/2$; for the sake of argument, therefore assume the former.

Now think of $V$ as a subset of $X_{i-1}$, and let $W$ denote the set of suffixes of $V$
of length $i-2$, and define $W'$ and ${}'W$ analogously to above. Then $\partial V = W'\backslash V$
by definition. Moreover, $|W'\backslash V|(2k-1)=|V'\backslash {}'V|$ since each element 
of $W'\backslash V$ can be prepended with $(2k-1)$ different letters to produce an element 
of $V'\backslash {}'V$. Since also $|V|(2k-1)=|V'|\ge |U|$ we deduce
$$h(X_{i-1}) \le \frac {|\partial V|}{|V|} = \frac {|W'\backslash V|}{|V|} = \frac {|V'\backslash {}'V|} {|V'|} \le
\frac {|V'\backslash {}'V|}{|U|} \le \frac {|V'\backslash U|}{|U|} = h(X_i)$$
\end{proof}

\subsection{Hyperbolic groups}

For an introduction to hyperbolic groups, see Gromov \cite{Gromov_hyperbolic}.
A finitely generated group $G$ is {\em hyperbolic} if it is coarsely negatively
curved on a large scale. This can be expressed in several equivalent ways in terms of
the geometry of the Cayley graph; the most useful characterizations are
\begin{enumerate}
\item{$\delta$-thinness of triangles;}
\item{a linear isoperimetric inequality; and}
\item{all asymptotic cones are $\R$-trees.}
\end{enumerate}
The adjective ``hyperbolic'' comes from the close (metric) resemblance to hyperbolic 
geometry. But there is another sense in which such groups are hyperbolic, namely 
in the dynamics of the (symbolic) geodesic flow.

Cannon showed \cite{Cannon} that in hyperbolic groups, a set of representative shortest
words in any given generating set can be enumerated by a finite state automaton. In the
language of digraphs, one version of Cannon's theorem can be expressed as follows.

Let $G$ be a hyperbolic group with a symmetric generating set $S$. Let $\Gamma$ be a finite
directed graph with a distinguished (initial) vertex, and edges labeled by elements of
$S$, in such a way that there is at most one edge with a given label emanating from each
vertex. A directed path $\gamma$ in $\Gamma$ starting at the initial vertex determines
a word $w(\gamma)$ in the generators $S$, and by evaluation, an element of $G$. Cannon
shows that one can find such a $\Gamma$ for which there is a $1$-$1$ correspondence between
such directed paths and elements of $G$, and moreover for which every word $w(\gamma)$ is
a geodesic --- i.e.\/ it is of shortest length among all words in $S^*$ representing a
given element of $G$. In more geometric terms, let $\til{\Gamma}$ denote the universal
cover of $\Gamma$ (it is also a directed graph), and let $\Gamma'$ be the subgraph of
$\til{\Gamma}$ which is the union of all directed rays starting at some lift of the initial
vertex. Then $\Gamma'$ embeds in the Cayley graph $C_S(G)$ in an edge-label respecting way
as a spanning tree, and every directed path in $\Gamma'$ is a geodesic in $C_S(G)$.

In this language, there is a correspondence between ``random'' words in $G$, and ``random''
directed walks in $\Gamma$. One thinks of $\Gamma$ as a topological Markov chain, and then
one can assign probabilities to the edges (the transitions between states) in a way which
maximizes the entropy. For such an assignment, the pushforward measure from walks of
length $n$ to the sphere of radius $n$ in $C_S(G)$ is coarsely equivalent to the uniform
measure on the sphere, and the limit as $n \to \infty$ converges to the Patterson-Sullivan
measure on the Gromov boundary $\partial G$ (see e.g.\/ Coornaert-Papadopoulos 
\cite{Coornaert_Papadopoulos}).

\medskip

A significant technical issue is that the graph $\Gamma$ is not typically {\em ergodic}.
Given a general directed graph $\Gamma$, one can form a new directed graph without
cycles, whose vertices are the ``communicating classes'' of vertices in $\Gamma$ (i.e. \/
equivalence classes of the relation $\sim$ where $u \sim v$ if there is a directed path
from $u$ to $v$ and another directed path from $v$ to $u$). Each vertex of the new
graph corresponds to an ergodic subgraph of $\Gamma$, whose adjacency matrix has a
real, non-negative (Perron-Frobenius) eigenvalue.

From the point of view of probability theory, only the vertices with maximal eigenvalue
are significant. It is an important consequence of a theorem of Coornaert \cite{Coornaert}
that for hyperbolic groups, such vertices do not occur in series, but only in parallel.
It follows that this maximal eigenvalue $\lambda$ is also the growth rate of the group;
i.e.\/ the unique $\lambda$ such that there are $\Theta(\lambda^n)$ words of length $n$.
The fact that such ``maximal'' vertices only occur in parallel means informally that
there are finitely many distinct classes $W$ so that all but $O(C^{-n^c})$ 
words of length $n$ fall into one of the classes of $W$, and for words
$v$ in a given class $W_i$, for each $\sigma$ of length $L\log(n)/\log(\lambda)$ with
$L<1$, there is some $f_i(\sigma)$ (depending only on $\sigma$ and on the class $W_i$) so that
$$\Pr\left( |C_\sigma(v) - nf_i(\sigma)| < n^{\epsilon+(1-L)/2} \right) = 1-O(C^{-n^c})$$
i.e.\/ the analogue of Proposition~\ref{small_error} holds for each class $W_i$
separately, and with essentially the same proof. This leaves two problems before one
can attempt to generalize the construction in \S~\ref{random_value_section} 
to arbitrary hyperbolic groups: one must be able to compare $f_i(\sigma)$ for different
classes $i$, and one must be able to compare $f_i(\sigma)$ with $f_i(\sigma^{-1})$.
These problems are largely solved by the methods of \cite{Calegari_Fujiwara, Calegari_Maher};
see especially \cite{Calegari_Maher} \S~3.7.

We believe that it should be straightforward (albeit technically involved)
to generalize the results of \S~\ref{random_value_section} to arbitrary hyperbolic groups,
and therefore feel confident in the following conjecture:

\begin{conjecture}
Let $G$ be a hyperbolic group with finite generating set $S$, and let $\lambda$ be
such that the number of elements of length $n$ is $\Theta(\lambda^n)$.
Let $v$ be a random element of word length $n$,
conditioned to lie in the commutator subgroup $[G,G]$. Then for any $\epsilon>0$ and
$C>1$,
$$|\scl(v)\log(n)/n - \log(\lambda)/6| \le \epsilon$$
with probability $1-O(n^{-C})$.
\end{conjecture}

A similar analogue of Theorem~\ref{random_norm_theorem} should also hold.

\subsection{Hyperbolic manifolds}

If $M$ is a closed hyperbolic $d$-manifold, it makes sense to study the stable commutator
length of random closed geodesics with length in $[n-\delta,n+\delta]$ for some 
fixed $\delta$ (conditioned to be homologically trivial). The geodesic flow on a hyperbolic
manifold is the canonical example of an Anosov flow, and the analogues of Lezaud's Chernoff-type
inequality are the mixing theorems of Pollicott \cite{Pollicott} and others.

The correct analogue of $\log(\lambda)$ should be the exponential 
growth rate of the number of orbits as a function of length which is just $d-1$ 
(i.e.\/ the volume entropy) where $d$ is the dimension. The following conjecture seems
very reasonable:

\begin{conjecture}
Let $M$ be a closed hyperbolic $d$-manifold. Fix some $\delta>0$.
Let $\gamma$ be a random geodesic of length in $[n-\delta,n+\delta]$ conditioned to
be homologically trivial, and let $v$ be the corresponding conjugacy class in $\pi_1(M)$.
Then for any $\epsilon>0$ and $C>1$,
$$|\scl(v)\log(n)/n - (d-1)/6| \le \epsilon$$
with probability $1-O(n^{-C})$.
\end{conjecture}

If true, this conjecture would say that one can recover (to any desired accuracy)
the {\em length} of a random geodesic
directly from the bounded cohomology of $\pi_1(M)$; this interpretation is obviously
very close to the spirit of Gromov's celebrated result discussed in the introduction.

\end{document}